\input ./style/arxiv-vmsta.cfg
\documentclass[numbers,compress,v1.0.1s]{vmsta}
\usepackage{vtexbibtags}
\usepackage{mathbh}

\volume{6}% Updated by VTEXPTS2LaTeX.exe, 14.10.2019 12:41
\issue{3}% Updated by VTEXPTS2LaTeX.exe, 14.10.2019 12:41
\pubyear{2019}
\firstpage{285}% Updated by VTEXPTS2LaTeX.exe, 14.10.2019 12:41
\lastpage{309}% Updated by VTEXPTS2LaTeX.exe, 14.10.2019 12:41
\aid{VMSTA136}% Updated by VTEXPTS2LaTeX.exe, 09.08.2019 11:38
\doi{10.15559/19-VMSTA136}% Updated by VTEXPTS2LaTeX.exe, 09.08.2019 11:38
\articletype{research-article}

\startlocaldefs
\ifdefined\HCode % tex4ht
\def\index#1{}
\else % LaTeX
\fi
\newtheorem{thm}{Theorem}
\newtheorem{lemma}{Lemma}
\newtheorem{propos}{Proposition}

\theoremstyle{definition}
\newtheorem{remark}{Remark}

\newcommand{\rrvert}{\vert}
\newcommand{\llvert}{\vert}
\allowdisplaybreaks
\endlocaldefs

\begin{document}

\begin{frontmatter}
\pretitle{Research Article}

\title{The risk model with stochastic premiums and a~multi-layer dividend strategy}

\author{\inits{O.}\fnms{Olena}~\snm{Ragulina}\ead[label=e1]{ragulina.olena@gmail.com}\orcid{0000-0002-5175-6620}}
\address{\institution{Taras Shevchenko National University of Kyiv},\break
Department of Probability Theory, Statistics and Actuarial Mathematics,\break
64 Volodymyrska Str., 01601 Kyiv, \cny{Ukraine}}

%\thankstext[id=f1]{}

%\dedicated{}

%\markboth{Authors}{Title}
\markboth{O. Ragulina}{The risk model with stochastic premiums and a~multi-layer dividend strategy}

\begin{abstract}
The paper deals with a generalization of the risk model with stochastic
premiums where dividends are paid according to a multi-layer dividend
strategy. First of all, we derive piecewise integro-differential
equations for the Gerber--Shiu function and the expected discounted
dividend payments until ruin. In addition, we concentrate on the
detailed investigation of the model in the case of exponentially
distributed claim and premium sizes and find explicit formulas for the
ruin probability as well as for the expected discounted dividend
payments. Lastly, numerical illustrations for some multi-layer dividend
strategies are presented.
\end{abstract}

\received{\sday{29} \smonth{3} \syear{2019}}% Updated by VTEXPTS2LaTeX.exe, 09.08.2019 11:38
\revised{\sday{4} \smonth{8} \syear{2019}}% Updated by VTEXPTS2LaTeX.exe, 09.08.2019 11:38
\accepted{\sday{4} \smonth{8} \syear{2019}}% Updated by VTEXPTS2LaTeX.exe, 09.08.2019 11:38
\publishedonline{\sday{28} \smonth{8} \syear{2019}}
\end{frontmatter}

%s1 #&#
\section{Introduction}%
\label{sec:1}
The ruin measures such as the ruin probability,\index{ruin probability} the surplus prior to
ruin and the deficit at ruin have attracted great interest of
researchers recently (see, e.g., \cite{AsAl2010,MiRa2016,RoScScTe1999} and references therein). Gerber and Shiu
\cite{GeSh1998} introduced the expected discounted penalty function for
the classical risk model, which enabled to study those risk measures
together by combining them into one function. After that, the so-called
Gerber--Shiu function\index{Gerber--Shiu function} has been investigated by many authors in more
general risk models\index{risk models} (see, e.g.,
\cite{BoCoLaMa2006,ChVr2014,ChTa2003,CoMaMa2008,CoMaMa2010,GeSh2005,He2014,Su2005,XiWu2006,YoXi2012,ZhYa2011_1}).

In particular, a lot of attention has been paid to the study of risk
models\index{risk models} where shareholders receive dividends from their insurance
company. De Finetti \cite{De1957}, who first considered dividend
strategies\index{dividend ! strategies} in insurance, dealt with a binomial model. For the classical
risk model and its different generalizations, different dividend
strategies\index{dividend ! strategies} have been studied in a number of papers (see, e.g.,
\cite{ChLi2011,CoMaMa2011,CoMaMa2014,La2008,LiWuSo2009,LiGa2004,LiPa2006,LiWiDr2003,LiLiPe2014,ShLiZh2013,Wa2015,YaZh2008}).
In addition, the monograph by Schmidli \cite{Sc2008} is devoted
to optimal dividend problems in insurance risk models.

Applying multi-layer dividend strategies enables to change the dividend
payment intensity depending on the current surplus. Albrecher and
Har\-tin\-ger \cite{AlHa2007} consider the modification of the
classical risk model where both the premium intensity and the dividend
payment intensity are assumed to be step functions depending on the
current surplus level. The authors derive algorithmic schemes for the
\xch{determination}{determanation} of explicit expressions for the Gerber--Shiu function\index{Gerber--Shiu function} and
the expected discounted dividend payments.\index{dividend ! payments} A similar risk model\index{risk models} is
considered by Lin and Sendova \cite{LiSe2008}, who derive a
piecewise integro-differential equation for the Gerber--Shiu function\index{Gerber--Shiu function}
and provide a recursive approach to obtain general solutions to that
equation and its generalizations. Developing a recursive algorithm to
calculate the moments of the expected discounted dividend payments\index{dividend ! payments} for
a class of risk models with Markovian claim arrivals, Badescu and
Landriault \cite{BaLa2008} generalize some of the results obtained
in \cite{AlHa2007} (see also \cite{BaDrLa2007} for some
results related to the class of Markovian risk models\index{risk models} with a multi-layer
dividend strategy).

The absolute ruin problem in the classical risk model with constant
interest force and a multi-layer dividend strategy is investigated in
\cite{YaZhLa2008}, where a piecewise integro-differential equation
for the discounted penalty function is derived, some explicit
expressions are given when claims are exponentially distributed and an
asymptotic formula for the absolute ruin probability is obtained for
heavy-tailed claim sizes. The dual model of the compound Poisson risk
model with a multi-layer dividend strategy under stochastic interest is
considered in \cite{Yi2012}. Results related to perturbed compound
Poisson risk models under multi-layer dividend strategies can be found
in \cite{MiSeTs2010,YaZh2009_2}. In addition, different classes
of more general renewal risk models are investigated in
\cite{DeZhDe2012,JiYaLi2012,YaZh2008,YaZh2009_1}, and some recent
papers deal with risk models\index{risk models} that incorporate various dependence
structures (see, e.g.,
\cite{LiMa2016,XiZo2017,ZhYa2011_2,ZhXiDe2015}).

The present paper generalizes the risk model with stochastic premiums\index{stochastic premiums}
introduced and investigated in \cite{Boi2003} (see also
\cite{MiRa2016}). In that risk model,\index{risk models} both claims and premiums are
modeled as compound Poisson processes,\index{compound Poisson processes} whereas premiums arrive with
constant intensity and are not random in the classical compound Poisson
risk model (see also \cite{MiRaSt2014,MiRaSt2015} for a
generalization of the classical risk model where an insurance company
gets additional funds whenever a claim arrives). In
\cite{Boi2003}, claim sizes and inter-claim times are assumed to be
mutually independent, and the same assumption is made concerning premium
arrivals. In contrast to \cite{Boi2003}, the recent paper
\cite{Ra2017} deals with the risk model with stochastic premiums\index{stochastic premiums} where
the dependence structures between claim sizes and inter-claim times as
well as premium sizes and inter-premium times are modeled by the
Farlie--Gumbel--Morgenstern copulas, and dividends are paid according
to a threshold dividend strategy. The Gerber--Shiu function, a special
case of which is the ruin probability,\index{ruin probability} and the expected discounted
dividend payments\index{dividend ! payments} until ruin are studied in \cite{Ra2017}. In the
present paper, we develop those results and make the assumption that
dividends are paid according to a multi-layer dividend strategy and all
random variables and processes are mutually independent.\looseness=1

The rest of the paper is organized as follows. In Section~\ref{sec:2},
we give a description of the risk model with stochastic premiums\index{stochastic premiums} and a
multi-layer dividend strategy. In Sections~\ref{sec:3} and~\ref{sec:4},
we derive piecewise integro-differential equations for the Gerber--Shiu\index{Gerber--Shiu function}
function and the expected discounted dividend payments\index{dividend ! payments} until ruin. Next,
in Section~\ref{sec:5}, we deal with exponentially distributed claim and
premium sizes and obtain explicit formulas for the ruin probability\index{ruin probability} and
the expected discounted dividend payments.\index{dividend ! payments} Finally, Section~\ref{sec:6}
provides some numerical illustrations.

%s2 #&#
\section{Description of the model}%
\label{sec:2}
Let $(\varOmega , \mathfrak{F}, \mathbb{P})$ be a probability space
satisfying the usual conditions, and let all the stochastic objects we
use below be defined on it.

In the risk model with stochastic premiums\index{stochastic premiums} introduced in
\cite{Boi2003} (see also \cite{MiRa2016}), claim sizes form a
sequence $(Y_{i})_{i\ge 1}$ of non-negative independent and identically
distributed (i.i.d.) random variables (r.v.'s) with cumulative
distribution function (c.d.f.) $F_{Y}(y)=\mathbb{P}[Y_{i}\le y]$, and
the number of claims on the time interval $[0,t]$ is a Poisson process
$(N_{t})_{t\ge 0}$ with constant intensity $\lambda >0$. In addition,
premium sizes form a sequence $(\bar{Y}_{i})_{i\ge 1}$ of non-negative
i.i.d. r.v.'s with c.d.f. $\bar{F}_{\bar{Y}}(y)=\mathbb{P}[\bar{Y}
_{i}\le y]$, and the number of premiums on the time interval
$[0,t]$ is a Poisson process $(\bar{N}_{t})_{t\ge 0}$ with constant
intensity $\bar{\lambda }>0$. Thus, the total claims and premiums on
$[0,t]$ equal $\sum_{i=1}^{N_{t}} Y_{i}$ and $\sum_{i=1}^{\bar{N}_{t}}
\bar{Y}_{i}$, respectively.

It is worth \xch{pointing}{to point} out that, here and subsequently, a sum is always set
to %be equal to
0 if the upper summation index is less than the lower
one. In particular, we have $\sum_{i=1}^{0} Y_{i} =0$ if $N_{t} =0$, and
$\sum_{i=1}^{0} \bar{Y}_{i} =0$ if $\bar{N}_{t} =0$. In what follows,
we also assume that the r.v.'s $(Y_{i})_{i\ge 1}$ and $(\bar{Y}_{i})_{i
\ge 1}$ have finite expectations $\mu >0$ and $\bar{\mu }>0$,
respectively. Furthermore, we suppose that $(Y_{i})_{i\ge 1}$,
$(\bar{Y}_{i})_{i\ge 1}$, $(N_{t})_{t\ge 0}$ and $(\bar{N}_{t})_{t
\ge 0}$ are mutually independent.

Next, we denote a non-negative initial surplus of the insurance company
by $x$, and let $X_{t}(x)$ be its surplus at time $t$ provided that the
initial surplus is $x$. Then the surplus process\index{surplus process} $  (X_{t}(x))_{t\ge 0}$ %follows
is defined by the \xch{equality}{equation}
%
%e1 #&#
\begin{equation}
\label{eq:1} X_{t}(x) =x+\sum_{i=1}^{\bar{N}_{t}}
\bar{Y}_{i} -\sum_{i=1}^{N_{t}}
Y _{i}, \quad t\ge 0.
\end{equation}

In contrast to the risk model\index{risk models} considered in \cite{Boi2003}, we
make the additional assumption that the insurance company pays dividends
to its shareholders according to a $k$-layer dividend strategy with
$k\ge 2$. Let $\mathbf{b}=(b_{1},\ldots ,b_{k-1})$ be a
$(k-1)$-dimensional vector with real-valued components such that
$0<b_{1}<\cdots <b_{k-1}<\infty $. Besides that, we set $b_{0}=0$ and
$b_{k}=\infty $. Let $  (X_{t}^{\mathbf{b}}(x)  )_{t\ge 0}$ denote
the modified surplus process\index{surplus process} under the $k$-layer dividend strategy
$\mathbf{b}$, which implies that dividends are paid continuously at a
rate $d_{j}>0$ whenever $b_{j-1}\le X_{t}^{\mathbf{b}}(x)< b_{j}$, i.e.
the process $  (X_{t}^{\mathbf{b}}(x)  )_{t\ge 0}$ is in the
$j$th layer at time $t$, where $1\le j\le k$. Then
%
%e2 #&#
\begin{equation}
\label{eq:2} X_{t}^{\mathbf{b}}(x) =x+\sum
_{i=1}^{\bar{N}_{t}} \bar{Y}_{i} -\sum
_{i=1}^{N_{t}} Y_{i} - \int
_{0}^{t} \sum_{j=1}^{k}
d_{j} \Eins \bigl( b_{j-1}\le X_{s}^{\mathbf{b}}(x)<
b_{j} \bigr)\, \mathrm{d}s, \quad t\ge 0,
\end{equation}
where $\Eins (\cdot )$ is the indicator function.

From now on, we suppose that the net profit condition\index{net profit condition} holds, which in
this case means that
%
%e3 #&#
\begin{equation}
\label{eq:3} \bar{\lambda }\bar{\mu }> \lambda \mu +\max_{1\le j\le k}
\{d_{j}\}.
\end{equation}

Let $(D_{t})_{t\ge 0}$ denote the dividend distributing process. For the
$k$-layer dividend strategy described above, we have
\begin{equation*}
\mathrm{d}D_{t} =d_{j}\, \mathrm{d}s \quad \text{if}\
 b_{j-1} \le X_{t}^{\mathbf{b}}(x)<
b_{j}, \quad 1\le j\le k.
\end{equation*}

Next, let $\tau _{\mathbf{b}}(x)= \inf \{t\ge 0\colon X_{t}^{
\mathbf{b}}(x) <0\}$ be the ruin time for the risk process $  (X
_{t}^{\mathbf{b}}(x)  )_{t\ge 0}$ defined by \eqref{eq:2}. In what
follows, we omit the dependence on $x$ and write $\tau _{\mathbf{b}}$
instead of $\tau _{\mathbf{b}}(x)$ when no confusion can arise.

For $\delta _{0} \ge 0$, the Gerber--Shiu function\index{Gerber--Shiu function} is defined by
\begin{equation*}
m(x,\mathbf{b}) =\mathbb{E} \bigl[ e^{-\delta _{0} \tau _{\mathbf{b}}} \, w\bigl(X_{\tau _{\mathbf{b}}-}^{\mathbf{b}}(x),
\bigl\llvert X_{\tau _{\mathbf{b}}} ^{\mathbf{b}}(x) \bigr\rrvert \bigr)\, \Eins (
\tau _{\mathbf{b}}<\infty ) \,|\, X _{0}^{\mathbf{b}}(x)=x
\bigr], \quad x\ge 0,
\end{equation*}
where $w(\cdot ,\cdot )$ is a bounded non-negative measurable function,
$X_{\tau _{\mathbf{b}}-}^{\mathbf{b}}(x)$ is the surplus immediately
before ruin and $|X_{\tau _{\mathbf{b}}}^{\mathbf{b}}(x)|$ is a deficit
at ruin. Note that if $w(\cdot ,\cdot ) \equiv 1$ and $\delta _{0}=0$,
then $m(x,\mathbf{b})$ becomes the infinite-horizon ruin probability
\begin{equation*}
\psi (x,\mathbf{b}) =\mathbb{E}\bigl[\Eins (\tau _{\mathbf{b}}<\infty ) \,|\,
X_{0}^{\mathbf{b}}(x)=x\bigr].
\end{equation*}

For $\delta >0$, the expected discounted dividend payments\index{dividend ! payments} until ruin
are defined by
\begin{equation*}
v(x,\mathbf{b}) =\mathbb{E} \Biggl[ \int_{0}^{\tau _{\mathbf{b}}}
e ^{-\delta t}\, \mathrm{d}D_{t} \,|\, X_{0}^{\mathbf{b}}(x)=x
\Biggr], \quad x\ge 0.
\end{equation*}

For simplicity of notation, we also write $m(x)$, $\psi (x)$ and
$v(x)$ instead of $m(x,\mathbf{b})$, $\psi (x,\mathbf{b})$ and
$v(x,\mathbf{b})$, respectively. For all $1\le j\le k$ and $b_{j-1}
\le x\le b_{j}$, we also set $m_{j}(x)=m(x,\mathbf{b})$, $\psi _{j}(x)=
\psi (x,\mathbf{b})$ and $v_{j}(x)=v(x,\mathbf{b})$. Thus, the functions
$m_{j}(x)$, $\psi _{j}(x)$ and $v_{j}(x)$ are defined on $[b_{j-1},b
_{j}]$, and we have $m_{j}(b_{j})=m_{j+1}(b_{j})$, $\psi _{j}(b_{j})=
\psi _{j+1}(b_{j})$ and $v_{j}(b_{j})=v_{j+1}(b_{j})$ for all
$1\le j\le k-1$.

%r1 #&#
\begin{remark}
\label{rem:1}Note that although we %imply
consider the interval $[b_{k-1},\infty )$ instead of
$[b_{j-1},b_{j}]$ if $j=k$, for the sake of convenience and compactness,
here and subsequently, we do write $[b_{j-1},b_{j}]$ for all
$1\le j\le k$. In addition, in what follows, the derivatives of all
functions at the ends of the closed intervals $[b_{j-1},b_{j}]$ are
assumed to be one-sided.
\end{remark}

%s3 #&#
\section{Piecewise integro-differential equation for the Gerber--Shiu function\index{Gerber--Shiu function}}%
\label{sec:3}
%
%t1 #&#
\begin{thm}
\label{thm:1}
Let the surplus process\index{surplus process} $  (X_{t}^{b}(x)  )_{t\ge 0}$ %follow
be defined by \eqref{eq:2} under the above assumptions, and let $F_{Y}(y)$ and
$w(u_{1},u_{2})$ be continuous on $\mathbb{R}_{+}$ and $\mathbb{R}
_{+}^{2}$, respectively. Then the function $m(x)$ is differentiable on
the intervals $[b_{j-1},b_{j}]$ for all $1\le j\le k$ and satisfies the
piecewise integro-differential equation
%
%e4 #&#
\begin{align}
\label{eq:4} %
%\begin{split}
&d_{j}
m'(x)+(\lambda +\bar{\lambda }+\delta _{0})m(x) =
\lambda \int_{0}^{x} m(x-y)\,
\mathrm{d}F_{Y}(y)\nonumber
\\[-2pt]
&\quad +\lambda \int_{x}^{\infty } w(x,y-x)\,
\mathrm{d}F_{Y}(y) +\bar{ \lambda }\int_{0}^{\infty }
m(x+y)\, \mathrm{d}F_{\bar{Y}}(y), \quad  x\in [b_{j-1},b_{j}].
%\end{split} %
\end{align}
\end{thm}

\begin{proof}
We now fix any $j$ such that $1\le j\le k$ and deal with the case
$x\in [b_{j-1},b_{j}]$. For all $x\in [b_{j-1},b_{j}]$, we define the
following functions:
\begin{gather*}
a_{1}(x)=(x-b_{j-1})/d_{j}+(b_{j-1}-b_{j-2})/d_{j-1}
+\cdots +(b_{2}-b _{1})/d_{2}+(b_{1}-b_{0})/d_{1},\\[-2pt]
a_{2}(x)=(x-b_{j-1})/d_{j}+(b_{j-1}-b_{j-2})/d_{j-1}+
\cdots +(b_{2}-b _{1})/d_{2},\\[-2pt]
\ldots\\[-2pt]
a_{j-1}(x)=(x-b_{j-1})/d_{j}+(b_{j-1}-b_{j-2})/d_{j-1},\\[-2pt]
a_{j}(x)=(x-b_{j-1})/d_{j}.
\end{gather*}

From these equalities we conclude that for all $x\in [b_{j-1},b_{j}]$,
the process\break  $  (X_{t}^{\mathbf{b}}(x)  )_{t\ge 0}$ up to its first
jump is in the $j$th layer if $t\in [0,a_{j}(x)]$ and in the $i$th layer
if $t\in [a_{i+1}(x),a_{i}(x)]$, where $1\le i\le j-1$. Thus, for any
$x\in [b_{j-1},b_{j}]$, the sequence $a_{j}(x)$, $a_{j-1}(x),\ldots, a_{1}(x)$ defines the times when $  (X_{t}^{\mathbf{b}}(x)
  )_{t\ge 0}$ passes through the values $b_{j-1}$, $b_{j-2},
\ldots , b_{0}$ provided that it has no jumps until those times.

It is easily seen that the time of the first jump of $  (X_{t}^{
\mathbf{b}}(x)  )_{t\ge 0}$ is exponentially distributed with mean
$1/(\lambda +\bar{\lambda })$. Considering the time and the size of the
first jump of that process and applying the law of total probability we
obtain
%
%e5 #&#
\begin{equation}
\label{eq:5} m(x)=I_{j}(x)+I_{j-1}(x)+\cdots
+I_{1}(x)+I_{0}(x), \quad x\in [b_{j-1},b
_{j}],
\end{equation}
where
\begin{equation*}
\begin{split} I_{j}(x) &=\int_{0}^{a_{j}(x)}
e^{-(\lambda +\bar{\lambda })t} \Biggl( \lambda \int_{0}^{x-d_{j} t}
e^{-\delta _{0} t}\, m (x-d_{j} t -y )\, \mathrm{d}F_{Y}(y)
\\[-2pt]
&\quad +\lambda \int_{x-d_{j} t}^{\infty }
e^{-\delta _{0} t}\, w (x-d_{j} t,\, y-x+d_{j} t )\,
\mathrm{d}F_{Y}(y)
\\[-2pt]
&\quad +\bar{\lambda }\int_{0}^{\infty }
e^{-\delta _{0} t}\, m (x-d _{j} t +y )\, \mathrm{d}F_{\bar{Y}}(y)
\Biggr) \mathrm{d}t,
\end{split} %
\end{equation*}
\begin{equation*}
\begin{split} I_{j-1}(x) &=\int_{a_{j}(x)}^{a_{j-1}(x)}
e^{-(\lambda +\bar{\lambda })t} \Biggl( \lambda \int_{0}^{b_{j-1}-d_{j-1} (t-a_{j}(x))}
e^{-\delta
_{0} t}
\\[-2pt]
& \qquad \times m \bigl(b_{j-1}-d_{j-1}
\bigl(t-a_{j}(x)\bigr) -y \bigr)\, \mathrm{d}F_{Y}(y)
\\[-2pt]
&\quad +\lambda \int_{b_{j-1}-d_{j-1}(t-\xch{a_{j}(x))}{a_{j}(x)}}^{\infty }
e^{-
\delta _{0} t}\, w \bigl(b_{j-1}-d_{j-1}
\bigl(t-a_{j}(x)\bigr),
\\[-2pt]
& \qquad y-b_{j-1}+d_{j-1} \bigl(t-a_{j}(x)
\bigr) \bigr)\, \mathrm{d}F_{Y}(y)
\\[-2pt]
&\quad +\bar{\lambda }\int_{0}^{\infty }
e^{-\delta _{0} t}\, m \bigl(b _{j-1}-d_{j-1}
\bigl(t-a_{j}(x)\bigr) +y \bigr)\, \mathrm{d}F_{\bar{Y}}(y)
\Biggr) \mathrm{d}t,\\[-2pt]
 &\qquad\qquad\qquad\qquad\qquad\ldots
\end{split} %
\end{equation*}
\begin{equation*}
\begin{split} I_{1}(x) &=\int_{a_{2}(x)}^{a_{1}(x)}
e^{-(\lambda +\bar{\lambda })t} \Biggl( \lambda \int_{0}^{b_{1}-d_{1} (t-a_{2}(x))}
e^{-\delta _{0} t}
\\[-2pt]
& \qquad \times m \bigl(b_{1}-d_{1}
\bigl(t-a_{2}(x)\bigr) -y \bigr)\, \mathrm{d}F_{Y}(y)
\\[-2pt]
&\quad +\lambda \int_{b_{1}-d_{1} (t-a_{2}(x))}^{\infty }
e^{-\delta
_{0} t}\, w \bigl(b_{1}-d_{1}
\bigl(t-a_{2}(x)\bigr),
\\[-2pt]
& \qquad y-b_{1}+d_{1} \bigl(t-a_{2}(x)
\bigr) \bigr)\, \mathrm{d}F_{Y}(y)
\\[-2pt]
&\quad +\bar{\lambda }\int_{0}^{\infty }
e^{-\delta _{0} t}\, m \bigl(b _{1}-d_{1}
\bigl(t-a_{2}(x)\bigr) +y \bigr)\, \mathrm{d}F_{\bar{Y}}(y)
\Biggr) \mathrm{d}t,
\end{split} %
\end{equation*}
\begin{equation*}
I_{0}(x)=e^{-(\lambda +\bar{\lambda }+\delta _{0})\, a_{1}(x)}\, w(0,0).
\end{equation*}

Note that the term $I_{i}(x)$, $1\le i\le j$, in \eqref{eq:5}
corresponds to the case where $  (X_{t}^{\mathbf{b}}(x)  )_{t
\ge 0}$ is in the $i$th layer when its first jump occurs, and the term
$I_{0}(x)$ corresponds to the case where there are no jumps of
$  (X_{t}^{\mathbf{b}}(x)  )_{t\ge 0}$ up to the time
$a_{1}(x)$.\looseness=1

Changing the variable $x-d_{j} t=s$ in the outer integral in the
expression for $I_{j}(x)$ yields
%
%e6 #&#
\begin{align}
\label{eq:6} %
%\begin{split}
I_{j}(x) &=
\frac{1}{d_{j}}\, e^{-(\lambda +\bar{\lambda }+\delta _{0})x/d
_{j}} \int_{b_{j-1}}^{x}
e^{(\lambda +\bar{\lambda }+\delta _{0})s/d
_{j}} \Biggl( \lambda \int_{0}^{s}
m(s-y)\, \mathrm{d}F_{Y}(y)\nonumber
\\[-2pt]
&\quad +\lambda \int_{s}^{\infty } w(s,y-s)\,
\mathrm{d}F_{Y}(y) +\bar{ \lambda }\int_{0}^{\infty }
m(s+y)\, \mathrm{d}F_{\bar{Y}}(y) \Biggr) \mathrm{d}s.
%\end{split}
%
\end{align}

Changing the variable $b_{j-1}-d_{j-1} (t-a_{j}(x))=s$ in the outer
integral in the expression for $I_{j-1}(x)$ gives
%
%e7 #&#
\begin{align}
\label{eq:7} %
%\begin{split}
I_{j-1}(x) &=
\frac{1}{d_{j-1}}\, e^{-(\lambda +\bar{\lambda }+\delta
_{0})(a_{j}(x)+b_{j-1}/d_{j-1})}\nonumber
\\[-2pt]
&\quad \times \int_{b_{j-2}}^{b_{j-1}}
e^{(\lambda +\bar{\lambda }+
\delta _{0})s/d_{j-1}} \Biggl( \lambda \int_{0}^{s}
m(s-y)\, \mathrm{d}F_{Y}(y)\nonumber
\\[-2pt]
& \qquad +\lambda \int_{s}^{\infty } w(s,y-s)\,
\mathrm{d}F_{Y}(y) +\bar{ \lambda }\int_{0}^{\infty }
m(s+y)\, \mathrm{d}F_{\bar{Y}}(y) \Biggr) \mathrm{d}s.
%\end{split}
%
\end{align}

In the same manner we change variables in all the outer integrals on the
right-hand side of \eqref{eq:5}. Finally, changing the variable
$b_{1}-d_{1} (t-a_{2}(x))=s$ in the outer integral in the expression for
$I_{1}(x)$ yields
%
%e8 #&#
\begin{align}
\label{eq:8} %
%\begin{split}
I_{1}(x) &=
\frac{1}{d_{1}}\, e^{-(\lambda +\bar{\lambda }+\delta _{0})(a
_{2}(x)+b_{1}/d_{1})} \int_{b_{0}}^{b_{1}}
e^{(\lambda +\bar{\lambda
}+\delta _{0})s/d_{1}}\nonumber
\\[-2pt]
&\quad \times \Biggl( \lambda \int_{0}^{s}
m(s-y)\, \mathrm{d}F_{Y}(y) +\lambda \int_{s}^{\infty }
w(s,y-s)\, \mathrm{d}F_{Y}(y)\nonumber
\\[-2pt]
& \qquad +\bar{\lambda }\int_{0}^{\infty } m(s+y)\,
\mathrm{d}F_{\bar{Y}}(y) \Biggr) \mathrm{d}s.
%\end{split} %
\end{align}

Thus, from the above and equality \eqref{eq:5} we deduce that
$m(x)$ is continuous on $[b_{j-1},b_{j}]$, and hence, on $\mathbb{R}
_{+}$. Therefore, \eqref{eq:6} implies that $I_{j}(x)$ is differentiable
on $[b_{j-1},b_{j}]$, and for all $x\in [b_{j-1},b_{j}]$, we get
\begin{equation*}
\begin{split} I'_{j}(x) &=-
\frac{\lambda +\bar{\lambda }+\delta _{0}}{d_{j}}\, I_{j}(x) +\frac{1}{d_{j}} \Biggl(
\lambda \int_{0}^{x} m(x-y)\, \mathrm{d}F
_{Y}(y)
\\[-2pt]
&\quad +\lambda \int_{x}^{\infty } w(x,y-x)\,
\mathrm{d}F_{Y}(y) +\bar{ \lambda }\int_{0}^{\infty }
m(x+y)\, \mathrm{d}F_{\bar{Y}}(y) \Biggr).
\end{split} %
\end{equation*}

Moreover, it is easily seen, e.g. from \eqref{eq:7} and \eqref{eq:8},
that all the functions $I_{j-1}(x), \ldots , \break I_{1}(x)$ and
$I_{0}(x)$ are differentiable on $[b_{j-1},b_{j}]$, and for all
$x\in [b_{j-1},b_{j}]$, we have
\begin{gather*}
I'_{j-1}(x)=-\frac{\lambda +\bar{\lambda }+\delta _{0}}{d_{j}}\,
I_{j-1}(x), \quad \ldots , \quad I'_{1}(x)=-
\frac{\lambda +\bar{
\lambda }+\delta _{0}}{d_{j}}\, I_{1}(x),
\\
I'_{0}(x)=-\frac{\lambda +\bar{\lambda }+\delta _{0}}{d_{j}}\,
I_{0}(x).
\end{gather*}

From \eqref{eq:5} it follows that $m(x)$ is also differentiable on
$[b_{j-1},b_{j}]$. Differentiating \eqref{eq:5} and taking into account
expressions for $I'_{j}(x)$, $I'_{j-1}(x), \ldots , I'_{1}(x)$,
$I'_{0}(x)$ we obtain
\begin{equation*}
\begin{split} &m'(x)=-\frac{\lambda +\bar{\lambda }+\delta _{0}}{d_{j}}\, m(x) +
\frac{1}{d
_{j}} \Biggl( \lambda \int_{0}^{x}
m(x-y)\, \mathrm{d}F_{Y}(y)
\\[-2pt]
&\quad +\lambda \int_{x}^{\infty } w(x,y-x)\,
\mathrm{d}F_{Y}(y) +\bar{ \lambda }\int_{0}^{\infty }
m(x+y)\, \mathrm{d}F_{\bar{Y}}(y) \Biggr), \quad x\in
[b_{j-1},b_{j}],
\end{split} %
\end{equation*}
from which equation \eqref{eq:4} follows immediately.
\end{proof}

%r2 #&#
\begin{remark}
\label{rem:2}
To solve equation \eqref{eq:4}, we use the following boundary
conditions. The first $k-1$ conditions are easily obtained from the
equality $m_{j}(b_{j})=m_{j+1}(b_{j})$ or, equivalently, $
\lim_{x\uparrow b_{j}} m(x) =\lim_{x\downarrow b_{j}} m(x)$ for all
$1\le j\le k-1$. In addition, for the ruin probability,\index{ruin probability} using standard
considerations (see, e.g.,
\cite{MiRa2016,MiRaSt2015,RoScScTe1999}) we can show that $
\lim_{x\to \infty } \psi (x) =0$ provided that the net profit condition\index{net profit condition}
holds. Finally, it is evident that $\psi (0)=1$. Although equation
\eqref{eq:4} is not solvable analytically in the general case, we can
find explicit expressions for the corresponding ruin probability\index{ruin probability} in the
case where claim and premium sizes are exponentially distributed (see
Section~\ref{sec:5}). The uniqueness of the required solution to
\eqref{eq:4} should be justified in each case.\looseness=1
\end{remark}

%r3 #&#
\begin{remark}
\label{rem:3}
In the assertion of Theorem \ref{thm:1}, we require the continuity of
$F_{Y}(y)$. Note that if $F_{Y}(y)$ has positive points of
discontinuity, then $m(x)$ may be not differentiable at some interior
points of the intervals $[b_{j-1},b_{j}]$, $1\le j\le k$ (for details,
see \cite{MiRa2016,MiRaSt2015}). Moreover, it is easily seen from
\eqref{eq:4} that $m(x)$ is not differentiable at $x=b_{j}$,
$1\le j\le k-1$, since its one-sided derivatives do not coincide at
those points.
\end{remark}

%s4 #&#
\section{Piecewise integro-differential equation for the expected discounted dividend payments\index{dividend ! payments} until ruin}%
\label{sec:4}
%
%t2 #&#
\begin{thm}
\label{thm:2}
Let the surplus process\index{surplus process} $  (X_{t}^{b}(x)  )_{t\ge 0}$ %follow
be defined by \eqref{eq:2} under the above assumptions, and let $F_{Y}(y)$ be
continuous on $\mathbb{R}_{+}$. Then the function $v(x)$ is
differentiable on the intervals $[b_{j-1},b_{j}]$ for all $1\le j
\le k$ and satisfies the piecewise integro-differential equation
%
%e9 #&#
\begin{align}
\label{eq:9} %
%\begin{split}
 &d_{j}
v'(x)+(\lambda +\bar{\lambda }+\delta )v(x) =\lambda \int
_{0} ^{x} v(x-y)\, \mathrm{d}F_{Y}(y)\nonumber
\\[-2pt]
&\quad +\bar{\lambda }\int_{0}^{\infty } v(x+y)\,
\mathrm{d}F_{
\bar{Y}}(y)+d_{j}, \quad x\in
[b_{j-1},b_{j}].
%\end{split} %
\end{align}
\end{thm}

\begin{proof}
We now fix any $j$ such that $1\le j\le k$ and deal with the case
$x\in [b_{j-1},b_{j}]$. As in the proof of Theorem~\ref{thm:1},
considering the time and the size of the first jump of $  (X_{t}
^{\mathbf{b}}(x)  )_{t\ge 0}$ and applying the law of total
probability we have
%
%e10 #&#
\begin{align}
\label{eq:10} %
%\begin{split}
v(x) &=I_{1,j}(x)
+I_{2,j}(x) +I_{1,j-1}(x) +I_{2,j-1}(x) +\cdots\nonumber
\\
&\quad +I_{1,1}(x) +I_{2,1}(x) +I_{1,0}(x),
\quad x\in [b_{j-1},b _{j}],
%\end{split} %
\end{align}
where
\begin{gather*}
I_{1,j}(x)=\int_{0}^{a_{j}(x)} (\lambda +
\bar{\lambda }) e^{-(\lambda
+\bar{\lambda })t} \Biggl( \int_{0}^{t}
d_{j} e^{-\delta s}\, \mathrm{d}s \Biggr)\mathrm{d}t,
\\
\begin{split} I_{2,j}(x) &=\int_{0}^{a_{j}(x)}
e^{-(\lambda +\bar{\lambda })t} \Biggl( \lambda \int_{0}^{x-d_{j} t}
e^{-\delta t}\, v (x-d_{j} t -y )\, \mathrm{d}F_{Y}(y)
\\
&\quad +\bar{\lambda }\int_{0}^{\infty }
e^{-\delta t}\, v (x-d _{j} t +y )\, \mathrm{d}F_{\bar{Y}}(y)
\Biggr) \mathrm{d}t,
\end{split} %
\\
I_{1,j-1}(x)=\int_{a_{j}(x)}^{a_{j-1}(x)}\! (
\lambda +\bar{\lambda }) e^{-(\lambda +\bar{\lambda })t} \Biggl( \int_{0}^{a_{j}(x)}
d_{j} e ^{-\delta s}\, \mathrm{d}s +\!\int
_{a_{j}(x)}^{t} d_{j-1} e^{-\delta
s}\,
\mathrm{d}s \Biggr)\mathrm{d}t,
\\
\begin{split} I_{2,j-1}(x) &=\int_{a_{j}(x)}^{a_{j-1}(x)}
e^{-(\lambda +\bar{\lambda
})t} \Biggl( \lambda \int_{0}^{b_{j-1}-d_{j-1} (t-a_{j}(x))}
e^{-
\delta t}
\\
& \qquad \times v \bigl(b_{j-1}-d_{j-1}
\bigl(t-a_{j}(x)\bigr) -y \bigr)\, \mathrm{d}F_{Y}(y)
\\
&\quad +\bar{\lambda }\int_{0}^{\infty }
e^{-\delta t}\, v \bigl(b _{j-1}-d_{j-1}
\bigl(t-a_{j}(x)\bigr) +y \bigr)\, \mathrm{d}F_{\bar{Y}}(y)
\Biggr) \mathrm{d}t,
\end{split} %
\\
\ldots
\\
\begin{split} I_{1,1}(x) &=\int_{a_{2}(x)}^{a_{1}(x)}
\! (\lambda +\bar{\lambda }) e ^{-(\lambda +\bar{\lambda })t} \Biggl( \int_{0}^{a_{j}(x)}
d_{j} e ^{-\delta s}\, \mathrm{d}s +\!\int
_{a_{j}(x)}^{a_{j-1}(x)} d_{j-1} e ^{-\delta s}
\, \mathrm{d}s +\cdots
\\
&\quad +\int_{a_{3}(x)}^{a_{2}(x)} d_{2}
e^{-\delta s}\, \mathrm{d}s +\int_{a_{2}(x)}^{t}
d_{1} e^{-\delta s}\, \mathrm{d}s \Biggr) \mathrm{d}t,
\end{split} %
\end{gather*}
\begin{gather*}
\begin{split} I_{2,1}(x) &=\int_{a_{2}(x)}^{a_{1}(x)}
e^{-(\lambda +\bar{\lambda })t} \Biggl( \lambda \int_{0}^{b_{1}-d_{1} (t-a_{2}(x))}
e^{-\delta t}
\\[-2pt]
& \qquad \times v \bigl(b_{1}-d_{1}
\bigl(t-a_{2}(x)\bigr) -y \bigr)\, \mathrm{d}F_{Y}(y)
\\[-2pt]
&\quad +\bar{\lambda }\int_{0}^{\infty }
e^{-\delta t}\, v \bigl(b _{1}-d_{1}
\bigl(t-a_{2}(x)\bigr) +y \bigr)\, \mathrm{d}F_{\bar{Y}}(y)
\Biggr) \mathrm{d}t,
\end{split} %
\\[-2pt]
\begin{split} I_{1,0}(x) &=\int_{a_{1}(x)}^{\infty }
\! (\lambda +\bar{\lambda }) e ^{-(\lambda +\bar{\lambda })t} \Biggl( \int_{0}^{a_{j}(x)}
d_{j} e ^{-\delta s}\, \mathrm{d}s +\!\int
_{a_{j}(x)}^{a_{j-1}(x)} d_{j-1} e ^{-\delta s}
\, \mathrm{d}s +\cdots
\\[-2pt]
&\quad +\int_{a_{3}(x)}^{a_{2}(x)} d_{2}
e^{-\delta s}\, \mathrm{d}s +\int_{a_{2}(x)}^{a_{1}(x)}
d_{1} e^{-\delta s}\, \mathrm{d}s \Biggr) \mathrm{d}t,
\end{split} %
\end{gather*}
and the functions $a_{1}(x)$, $a_{2}(x), \ldots ,a_{j}(x)$ are
defined in the proof of Theorem~\ref{thm:1}.

Note that the terms $I_{1,i}(x)$ and $I_{2,i}(x)$, $1\le i\le j$, in
\eqref{eq:10} correspond to the case where $  (X_{t}^{\mathbf{b}}(x)
  )_{t\ge 0}$ is in the $i$th layer when its first jump occurs, and
the term $I_{1,0}(x)$ corresponds to the case where there are no jumps
of $  (X_{t}^{\mathbf{b}}(x)  )_{t\ge 0}$ up to the time
$a_{1}(x)$. The terms $I_{1,i}(x)$, $0\le i\le j$, are equal to the
discounted dividend payments\index{dividend ! payments} until the first jump of $  (X_{t}^{
\mathbf{b}}(x)  )_{t\ge 0}$ provided that the process is in the
$i$th layer, whereas the terms $I_{2,i}(x)$, $1\le i\le j$, are equal
to the corresponding expected discounted dividend payments\index{dividend ! payments} after that
time.

Next, we set
%
%e11 #&#
\begin{equation}
\label{eq:11} I_{1,*}(x)=I_{1,j}(x)
+I_{1,j-1}(x) +\cdots +I_{1,1}(x) +I_{1,0}(x),
\quad x\in [b_{j-1},b_{j}].
\end{equation}

Thus, $I_{1,*}(x)$ describes the expected discounted dividend payments\index{dividend ! payments}
until the first jump of $  (X_{t}^{\mathbf{b}}(x)  )_{t\ge 0}$.

Rearranging terms in the expression for $I_{1,*}(x)$ gives
%
%e12 #&#
\begin{align}
\label{eq:12} %
%\begin{split}
I_{1,*}(x) &=(\lambda +
\bar{\lambda }) \Biggl(\, \int_{0}^{a_{j}(x)} e
^{-(\lambda +\bar{\lambda })t} \Biggl( \int_{0}^{t}
d_{j} e^{-\delta
s}\, \mathrm{d}s \Biggr)\mathrm{d}t\nonumber
\\[-2pt]
&\quad +\int_{a_{j}(x)}^{a_{j-1}(x)} e^{-(\lambda +\bar{\lambda })t}
\Biggl( \int_{a_{j}(x)}^{t} d_{j-1}
e^{-\delta s}\, \mathrm{d}s \Biggr)\mathrm{d}t +\cdots\nonumber
\\[-2pt]
&\quad +\int_{a_{2}(x)}^{a_{1}(x)} e^{-(\lambda +\bar{\lambda })t}
\Biggl( \int_{a_{2}(x)}^{t} d_{1}
e^{-\delta s}\, \mathrm{d}s \Biggr) \mathrm{d}t\nonumber
\\[-2pt]
&\quad +\int_{a_{j}(x)}^{\infty } e^{-(\lambda +\bar{\lambda })t}
\Biggl( \int_{0}^{a_{j}(x)} d_{j}
e^{-\delta s}\, \mathrm{d}s \Biggr) \mathrm{d}t\nonumber
\\[-2pt]
&\quad +\int_{a_{j-1}(x)}^{\infty } e^{-(\lambda +\bar{\lambda })t}
\Biggl( \int_{a_{j}(x)}^{a_{j-1}(x)} d_{j-1}
e^{-\delta s}\, \mathrm{d}s \Biggr)\mathrm{d}t +\cdots\nonumber
\\[-2pt]
&\quad +\int_{a_{1}(x)}^{\infty } e^{-(\lambda +\bar{\lambda })t}
\Biggl( \int_{a_{2}(x)}^{a_{1}(x)} d_{1}
e^{-\delta s}\, \mathrm{d}s \Biggr) \mathrm{d}t \, \Biggr).
%\end{split}
%
\end{align}

Taking all the integrals on the right-hand side of \eqref{eq:12} and
simplifying the resulting expression we get
%
%e13 #&#
\begin{align}
\label{eq:13} %
%\begin{split}
I_{1,*}(x) &=
\frac{d_{j}}{\lambda +\bar{\lambda }+\delta }\, \bigl(1-e ^{-(\lambda +\bar{\lambda }+\delta )\,a_{j}(x)} \bigr)\nonumber
\\
&\quad +\frac{d_{j-1}}{\lambda +\bar{\lambda }+\delta }\, \bigl(e ^{-(\lambda +\bar{\lambda }+\delta )\,a_{j}(x)} -e^{-(\lambda +\bar{
\lambda }+\delta )\,a_{j-1}(x)}
\bigr) +\cdots\nonumber
\\
&\quad +\frac{d_{1}}{\lambda +\bar{\lambda }+\delta }\, \bigl(e^{-(
\lambda +\bar{\lambda }+\delta )\,a_{2}(x)} -e^{-(\lambda +\bar{
\lambda }+\delta )\,a_{1}(x)}
\bigr).
%\end{split} %
\end{align}

Changing the variable $x-d_{j} t=s$ in the outer integral in the
expression for $I_{2,j}(x)$ gives
%
%e14 #&#
\begin{align}
\label{eq:14} %
%\begin{split}
I_{2,j}(x) &=
\frac{1}{d_{j}}\, e^{-(\lambda +\bar{\lambda }+\delta )x/d
_{j}} \int_{b_{j-1}}^{x}
e^{(\lambda +\bar{\lambda }+\delta )s/d_{j}}\nonumber
\\[-2pt]
&\quad \times \Biggl( \lambda \int_{0}^{s}
v(s-y)\, \mathrm{d}F_{Y}(y) +\bar{\lambda }\int
_{0}^{\infty } v(s+y)\, \mathrm{d}F_{\bar{Y}}(y)
\Biggr) \mathrm{d}s.
%\end{split} %
\end{align}

Likewise, changing the variable $b_{j-1}-d_{j-1} (t-a_{j}(x))=s$ in the
outer integral in the expression for $I_{2,j-1}(x)$ yields
%
%e15 #&#
\begin{align}
\label{eq:15} %
%\begin{split}
I_{2,j-1}(x) &=
\frac{1}{d_{j-1}}\, e^{-(\lambda +\bar{\lambda }+\delta
)(a_{j}(x)+b_{j-1}/d_{j-1})} \int_{b_{j-2}}^{b_{j-1}}
e^{(\lambda +\bar{
\lambda }+\delta )s/d_{j-1}}\nonumber
\\[-2pt]
& \qquad \times \Biggl( \lambda \int_{0}^{s}
v(s-y)\, \mathrm{d}F_{Y}(y) +\bar{ \lambda }\int
_{0}^{\infty } v(s+y)\, \mathrm{d}F_{\bar{Y}}(y)
\Biggr) \mathrm{d}s.
%\end{split} %
\end{align}

Next, in the same manner we change variables in all those outer
integrals on the right-hand side of \eqref{eq:10} that are not included
in the sum \eqref{eq:11}. Eventually, changing the variable
$b_{1}-d_{1} (t-a_{2}(x))=s$ in the outer integral in the expression for
$I_{2,1}(x)$ we obtain
%
%e16 #&#
\begin{align}
\label{eq:16} %
%\begin{split}
I_{2,1}(x) &=
\frac{1}{d_{1}}\, e^{-(\lambda +\bar{\lambda }+\delta )(a
_{2}(x)+b_{1}/d_{1})} \int_{b_{0}}^{b_{1}}
e^{(\lambda +\bar{\lambda
}+\delta )s/d_{1}}\nonumber
\\[-2pt]
&\quad \times \Biggl( \lambda \int_{0}^{s}
v(s-y)\, \mathrm{d}F_{Y}(y) +\bar{\lambda }\int
_{0}^{\infty } v(s+y)\, \mathrm{d}F_{\bar{Y}}(y)
\Biggr) \mathrm{d}s.
%\end{split} %
\end{align}

From \eqref{eq:13} it follows immediately that $I'_{1,*}(x)$ is
differentiable on $[b_{j-1},b_{j}]$, and for all $x\in [b_{j-1},b_{j}]$,
we get
\begin{equation*}
I'_{1,*}(x)=-\frac{\lambda +\bar{\lambda }+\delta }{d_{j}}\, \biggl( I
_{1,*}(x) -\frac{d_{j}}{\lambda +\bar{\lambda }+\delta } \biggr).
\end{equation*}

Next, from the above and equality \eqref{eq:10} we conclude that
$v(x)$ is continuous on $[b_{j-1},b_{j}]$, and hence, on $\mathbb{R}
_{+}$. Hence, \eqref{eq:14} implies that $I_{2,j}(x)$ is differentiable
on $[b_{j-1},b_{j}]$, and for all $x\in [b_{j-1},b_{j}]$, we have
\begin{equation*}
\begin{split}
I'_{2,j}(x) &=-
\frac{\lambda +\bar{\lambda }+\delta }{d_{j}}\, I_{2,j}(x)
\\[-2pt]
&\quad +\frac{1}{d_{j}} \Biggl( \lambda \int_{0}^{x}
v(x-y)\, \mathrm{d}F_{Y}(y) +\bar{\lambda }\int
_{0}^{\infty } v(x+y)\, \mathrm{d}F_{\bar{Y}}(y)
\Biggr).
\end{split} %
\end{equation*}

Furthermore, it follows immediately, e.g. from \eqref{eq:15} and
\eqref{eq:16}, that all the functions $I_{2,j-1}(x),\ldots , I_{2,1}(x)$
are differentiable on $[b_{j-1},b_{j}]$, and for all
$x\in [b_{j-1},b_{j}]$, we obtain
\begin{equation*}
I'_{2,j-1}(x)=-\frac{\lambda +\bar{\lambda }+\delta }{d_{j}}\,
I_{2,j-1}(x), \quad \ldots , \quad I'_{2,1}(x)=-
\frac{\lambda +\bar{
\lambda }+\delta }{d_{j}}\, I_{2,1}(x).
\end{equation*}

By \eqref{eq:10}, we conclude that $v(x)$ is also differentiable on
$[b_{j-1},b_{j}]$. Differentiating \eqref{eq:10} and taking into account
expressions for $I'_{1,*}(x)$, $I'_{2,j}(x)$, $I'_{2,j-1}(x),
\ldots , I'_{2,1}(x)$ we get
\begin{equation*}
\begin{split}
v'(x) &=-\frac{\lambda +\bar{\lambda }+\delta }{d_{j}}\, v(x) +1
+\frac{1}{d
_{j}} \Biggl( \lambda \int_{0}^{x}
v(x-y)\, \mathrm{d}F_{Y}(y)
\\[-2pt]
&\quad +\bar{\lambda }\int_{0}^{\infty } v(x+y)\,
\mathrm{d}F_{
\bar{Y}}(y) \Biggr), \quad x\in [b_{j-1},b_{j}],
\end{split} %
\end{equation*}
which immediately yields equation \eqref{eq:9}.
\end{proof}

%r4 #&#
\begin{remark}
\label{rem:4}
To solve equation \eqref{eq:9}, we obtain the first $k-1$ boundary
conditions from the equality $v_{j}(b_{j})=v_{j+1}(b_{j})$ or,
equivalently, $\lim_{x\uparrow b_{j}} v(x) =\lim_{x\downarrow b_{j}}
v(x)$ for all $1\le j\le k-1$. Moreover, if the net profit condition\index{net profit condition}
holds, applying arguments similar to those in
\cite[p.~70]{Sc2008} we can show that $\lim_{x\to \infty } v(x) =d
_{k}/{\delta }$. Lastly, it is easily seen that $v(0)=0$. The uniqueness
of the required solution to \eqref{eq:9} should be justified in each
case. Explicit expressions for $v(x)$ in the case where claim and
premium sizes are exponentially distributed are given in
\xch{Section~\ref{sec:5}.}{Section~\ref{sec:5}).}
\end{remark}

%r5 #&#
\begin{remark}
\label{rem:5}
If $F_{Y}(y)$ has positive points of discontinuity, then $v(x)$ may be
not differentiable at some interior points of the intervals
$[b_{j-1},b_{j}]$, $1\le j\le k$. Furthermore, from \eqref{eq:9} we
deduce that $v(x)$ is not differentiable at $x=b_{j}$, $1\le j\le k-1$.
\end{remark}

%s5 #&#
\section{Exponentially distributed claim and premium sizes}%
\label{sec:5}
In this section, we concentrate on the case where claim and premium
sizes are exponentially distributed, i.e.
%
%e17 #&#
\begin{equation}
\label{eq:17} f_{Y}(y) =\frac{1}{\mu }\,
e^{-y/\mu } \quad \text{and} \quad f_{
\bar{Y}}(y) =
\frac{1}{\bar{\mu }}\, e^{-y/\bar{\mu }}, \quad y\ge 0.
\end{equation}

%s5.1 #&#
\subsection{Explicit formulas for the ruin probability\index{ruin  probability}}%
\label{sec:5.1}
Let now $w(\cdot ,\cdot ) \equiv 1$ and $\delta _{0}=0$. Taking into
account \eqref{eq:17}, equation \eqref{eq:4} for the ruin probability\index{ruin probability} can
be written as
%
%e18 #&#
\begin{align}
\label{eq:18} %
%\begin{split}
&d_{j} \psi
'(x)+(\lambda +\bar{\lambda })\psi (x)\nonumber
\\[-2pt]
&\quad =\frac{\lambda }{\mu }\, e^{-x/\mu } \int_{0}^{x}
\psi (u)e ^{u/\mu }\, \mathrm{d}u+ \lambda e^{-x/\mu } +
\frac{\bar{\lambda }}{\bar{
\mu }}\, e^{x/\bar{\mu }} \int_{x}^{\infty }
\psi (u)e^{-u/\bar{
\mu }}\, \mathrm{d}u
%\end{split} %
\end{align}
for all $x\in [b_{j-1},b_{j}]$ and $1\le j\le k$.

We now reduce piecewise integro-differential equation \eqref{eq:18} to
a piecewise linear differential equation with constant coefficients.

%l1 #&#
\begin{lemma}
\label{lem:1}
Let the surplus process\index{surplus process} $  (X_{t}^{\mathbf{b}}(x)  )_{t\ge 0}$
%follow
be defined by \eqref{eq:2} under the above assumptions, and let claim and
premium sizes be exponentially distributed with means $\mu $ and
$\bar{\mu }$, respectively. Then $\psi (x)$ is a solution to the
piecewise differential equation
%
%e19 #&#
\begin{equation}
\label{eq:19} d_{j}\mu \bar{\mu }\psi '''(x)
+ \bigl( d_{j}(\bar{\mu }-\mu ) +\mu \bar{ \mu }(\lambda +\bar{
\lambda }) \bigr)\psi ''(x) +(\bar{\lambda }\bar{
\mu }-\lambda \mu -d_{j})\psi '(x)=0
\end{equation}
for all $x\in [b_{j-1},b_{j}]$ and $1\le j\le k$.
\end{lemma}

\begin{proof}
It is easily seen that the right-hand side of \eqref{eq:18} is
differentiable on $[b_{j-1},b_{j}]$. Therefore, $\psi (x)$ is twice
differentiable on $[b_{j-1},b_{j}]$. Differentiating \eqref{eq:18} gives
%
%e20 #&#
\begin{align}
\label{eq:20} %
%\begin{split}
&d_{j} \psi
''(x) +(\lambda +\bar{\lambda })\psi
'(x) =-\frac{1}{
\mu } \Biggl( \frac{\lambda }{\mu }\,
e^{-x/\mu } \int_{0}^{x} \psi (u)e
^{u/\mu }\, \mathrm{d}u +\lambda e^{-x/\mu } \Biggr)\nonumber
\\[-2pt]
&\quad +\frac{\bar{\lambda }}{\bar{\mu }^{2}}\, e^{x/\bar{\mu }} \int_{x}^{\infty }
\psi (u) e^{-u/\bar{\mu }}\, \mathrm{d}u + \biggl( \frac{
\lambda }{\mu } -
\frac{\bar{\lambda }}{\bar{\mu }} \biggr)\psi (x), \quad x\in [b_{j-1},b_{j}].
%\end{split} %
\end{align}

Multiplying \eqref{eq:20} by $\mu $ and adding \eqref{eq:18} we get
%
%e21 #&#
\begin{align}
\label{eq:21} %
%\begin{split}
&d_{j}\mu \psi
''(x) + \bigl(d_{j}+\mu (\lambda +
\bar{\lambda }) \bigr) \psi '(x) +\bar{\lambda } \biggl(1+
\frac{\mu }{\bar{\mu }} \biggr) \psi (x)\nonumber
\\[-2pt]
&\quad =\frac{\bar{\lambda }}{\bar{\mu }} \biggl(1+\frac{\mu }{\bar{
\mu }} \biggr)
e^{x/\bar{\mu }} \int_{x}^{\infty } \psi (u)
e^{-u/\bar{
\mu }}\, \mathrm{d}u, \quad x\in [b_{j-1},b_{j}].
%\end{split} %
\end{align}

From \eqref{eq:21} it follows that $\psi (x)$ has the third derivative
on $x\in [b_{j-1},b_{j}]$. Differentiating \eqref{eq:21} yields
%
%e22 #&#
\begin{align}
\label{eq:22} %
%\begin{split}
&d_{j}\mu \psi
'''(x) + \bigl(d_{j}+
\mu (\lambda +\bar{\lambda }) \bigr)\psi ''(x) +
\bar{\lambda } \biggl(1+\frac{\mu }{\bar{\mu }} \biggr)\psi '(x)\nonumber
\\[-2pt]
&=\frac{\bar{\lambda }}{\bar{\mu }^{2}} \biggl(1+\frac{\mu }{\bar{
\mu }} \biggr)
e^{x/\bar{\mu }} \int_{x}^{\infty }\! \psi (u)
e^{-u/\bar{
\mu }}\, \mathrm{d}u -\frac{\bar{\lambda }}{\bar{\mu }} \biggl(1+
\frac{
\mu }{\bar{\mu }} \biggr)\psi (x), \quad  x\in [b_{j-1},b_{j}].
%\end{split} %
\end{align}

Finally, multiplying \eqref{eq:22} by $(-\bar{\mu })$ and adding
\eqref{eq:21} we obtain \eqref{eq:19}.
\end{proof}

For $1\le j\le k$, we now define the following constants, which are used
in the assertion of Theorem \ref{thm:3} below:
\begin{equation*}
\mathrm{D}_{j}= \bigl( d_{j} (\mu +\bar{\mu }) +\mu
\bar{\mu }(\lambda -\bar{\lambda }) \bigr)^{2} +4\lambda \bar{\lambda
}\mu ^{2} \bar{ \mu }^{2},
\end{equation*}
\begin{equation*}
z_{1,j}=\frac{-  ( d_{j}(\bar{\mu }-\mu ) +\mu \bar{\mu }(\lambda
+\bar{\lambda })   ) +\sqrt{\mathrm{D}_{j}}}{2d_{j}\mu \bar{
\mu }}
\end{equation*}
and
\begin{equation*}
z_{2,j}=\frac{-  ( d_{j}(\bar{\mu }-\mu ) +\mu \bar{\mu }(\lambda
+\bar{\lambda })   ) -\sqrt{\mathrm{D}_{j}}}{2d_{j}\mu \bar{
\mu }}.
\end{equation*}

%t3 #&#
\begin{thm}
\label{thm:3}
Let the surplus process\index{surplus process} $  (X_{t}^{\mathbf{b}}(x)  )_{t\ge 0}$
%follow
be defined by \eqref{eq:2} under the above assumptions, and let claim and
premium sizes\vadjust{\goodbreak} be exponentially distributed with means $\mu $ and
$\bar{\mu }$, respectively. If the net profit condition\index{net profit condition} \eqref{eq:3}
holds, then we have
\begingroup
\abovedisplayskip=8pt
\belowdisplayskip=8pt
%e23 #&#
\begin{equation}
\label{eq:23} \psi _{j}(x)=C_{1,j}
e^{z_{1,j} x} +C_{2,j} e^{z_{2,j} x} +C_{3,j}
\end{equation}
for all $x\in [b_{j-1},b_{j}]$ and $1\le j\le k$, where $C_{3,k}=0$ and
all the other constants $C_{1,j}$, $C_{2,j}$ and $C_{3,j}$ are
determined from the system of linear equations
\eqref{eq:24}--\eqref{eq:27}:
%
%e24 #&#
\begin{align}
\label{eq:24} %
%\begin{split}
&\lambda e^{-b_{j-1}/\mu } \sum
_{l=1}^{j-1} \Biggl( \sum
_{i=1}^{2} \frac{C
_{i,l}}{\mu z_{i,l}+1}\,
\bigl(e^{(z_{i,l}+1/\mu )b_{l}} -e^{(z_{i,l}+1/
\mu )b_{l-1}} \bigr)\nonumber
\\[-2pt]
&\; + \bigl(e^{b_{l}/\mu } \!-\! e^{b_{l-1}/\mu } \bigr)
C_{3,l} \Biggr) \!+\!\sum_{i=1}^{2}
\biggl( \frac{\bar{\lambda }e^{b_{j-1}/\bar{
\mu }}}{\bar{\mu }z_{i,j}-1}\, \bigl(e^{(z_{i,j}-1/\bar{\mu })b_{j}} \!-\! e^{(z_{i,j}-1/\bar{\mu })b_{j-1}}
\bigr)\nonumber
\\[-2pt]
&\; -(d_{j} z_{i,j} +\lambda +\bar{\lambda })
e^{z_{i,j}b_{j-1}} \biggr) C_{i,j} - \bigl(\bar{\lambda
}e^{(b_{j-1}-b_{j})/\bar{\mu }} + \lambda \bigr) C_{3,j}\nonumber
\\[-2pt]
&\; +\bar{\lambda }e^{b_{j-1}/\bar{\mu }} \sum_{l=j+1}^{k}
\Biggl( \sum_{i=1}^{2}
\frac{C_{i,l}}{\bar{\mu }z_{i,l}-1}\, \bigl(e^{(z_{i,l}-1/\bar{
\mu })b_{l}} -e^{(z_{i,l}-1/\bar{\mu })b_{l-1}} \bigr)\nonumber
\\[-2pt]
&\; - \bigl(e^{-b_{l}/\bar{\mu }} -e^{-b_{l-1}/\bar{\mu }} \bigr) C _{3,l}
\Biggr) =-\lambda e^{-b_{j-1}/\mu }, \quad 1\le j\le k,
%\end{split} %
\end{align}
%
%e25 #&#
\begin{equation}
\label{eq:25} C_{1,1} +C_{2,1} +C_{3,1}=1,
\end{equation}
%
%e26 #&#
\begin{equation}
\label{eq:26} d_{j} \sum_{i=1}^{2}
z_{i,j} e^{z_{i,j} b_{j}} C_{i,j} -d_{j+1}
\sum_{i=1}^{2} z_{i,j+1}
e^{z_{i,j+1} b_{j}} C_{i,j+1}=0, \quad  1\le j\le k-1,
\end{equation}
and
%
%e27 #&#
\begin{equation}
\label{eq:27} \sum_{i=1}^{2}
e^{z_{i,j} b_{j}} C_{i,j} +C_{3,j} -\sum
_{i=1}^{2} e ^{z_{i,j+1} b_{j}}
C_{i,j+1} -C_{3,j+1}=0, \quad 1\le j\le k-1,
\end{equation}
\endgroup
provided that its determinant is not equal to 0.
\end{thm}

\begin{proof}
Taking into account the notation introduced in Section \ref{sec:2} and
applying\break  Lemma~\ref{lem:1} we conclude that the function $\psi _{j}(x)$
is a solution to \eqref{eq:19} on $x\in [b_{j-1},b_{j}]$ for each
$1\le j\le k$. The corresponding characteristic equation has the form
%
%e28 #&#
\begin{align}
\label{eq:28} %
%\begin{split}
&d_{j} \mu \bar{\mu
}z^{3} + \bigl( d_{j}(\bar{\mu }-\mu ) +\mu \bar{ \mu
}(\lambda +\bar{\lambda }) \bigr) z^{2} +(\bar{\lambda }\bar{ \mu }-
\lambda \mu -d_{j}) z =0
%\end{split} %
\end{align}
for all $1\le j\le k$. We first show that the equation
%
%e29 #&#
\begin{align}
\label{eq:29} %
%\begin{split}
&d_{j} \mu \bar{\mu
}z^{2} + \bigl( d_{j}(\bar{\mu }-\mu ) +\mu \bar{ \mu
}(\lambda +\bar{\lambda }) \bigr) z +(\bar{\lambda }\bar{\mu }- \lambda \mu
-d_{j})=\xch{0}{0.}
%\end{split} %
\end{align}
has two negative roots. Indeed, its discriminant is equal to
\begin{align*}
%
%\begin{split}
& \bigl( d_{j}(\bar{\mu }-\mu ) +\mu
\bar{\mu }(\lambda +\bar{\lambda }) \bigr)^{2} -4d_{j}
\mu \bar{\mu }(\bar{\lambda }\bar{\mu }-\lambda \mu -d_{j})
\\
& \:\: =d_{j}^{2} (\bar{\mu }-\mu )^{2}
+\mu ^{2} \bar{\mu }^{2} (\lambda +\bar{ \lambda
})^{2} +2d_{j} \mu \bar{\mu }(\lambda +\bar{\lambda })
(\bar{ \mu }-\mu )\\
&\:\: \quad  +4d_{j} \mu \bar{\mu }(d_{j} +
\lambda \mu -\bar{\lambda }\bar{ \mu })
\\
& \:\: =d_{j}^{2} (\mu +\bar{\mu })^{2}
+\mu ^{2} \bar{\mu }^{2} (\lambda +\bar{ \lambda
})^{2} +2d_{j} \mu \bar{\mu }(\lambda -\bar{\lambda })
( \mu +\bar{\mu })
\\
& \:\: = \bigl( d_{j} (\mu +\bar{\mu }) +\mu \bar{\mu }(\lambda -
\bar{\lambda }) \bigr)^{2} +4\lambda \bar{\lambda }\mu
^{2} \bar{\mu }^{2}.
%\end{split} %
\end{align*}
Hence, it is positive and coincides with the constant $\mathrm{D}_{j}$
defined above. Consequently, $z_{1,j}$ and $z_{2,j}$ defined before the
assertion of the theorem are two real roots of equation \eqref{eq:29}.
By the net profit condition\index{net profit condition} \eqref{eq:3}, we have
\begin{equation*}
\bar{\lambda }\bar{\mu }-\lambda \mu -d_{j} >0
\end{equation*}
and
\begin{equation*}
d_{j}(\bar{\mu }-\mu ) +\mu \bar{\mu }(\lambda +\bar{\lambda }) =
\mu (\bar{\lambda }\bar{\mu }-\lambda \mu -d_{j}) +\lambda \mu
^{2} + \lambda \mu \bar{\mu }+d_{j}\bar{\mu }>0
\end{equation*}
for all $1\le j\le k$, which implies that $z_{1,j}<0$ and $z_{2,j}<0$.

Therefore, $z_{1,j}<0$, $z_{2,j}<0$ and $z_{3,j}=0$ are roots of
equation \eqref{eq:28}, and we get \eqref{eq:23} with some constants
$C_{1,j}$, $C_{2,j}$ and $C_{3,j}$. Moreover, since condition
\eqref{eq:3} holds, using standard considerations (see, e.g.,
\cite{MiRa2016,MiRaSt2015,RoScScTe1999}) we can easily show that
$\lim_{x\to \infty } \psi (x) =0$, which yields $C_{3,k}=0$.

To determine all the other constants $C_{1,j}$, $C_{2,j}$ and
$C_{3,j}$, we use the following boundary conditions. The first $k$
conditions are obtained by letting $x=b_{j-1}$ in \eqref{eq:18} for
$1\le j\le k$:
%
%e30 #&#
\begin{align}
\label{eq:30} %
%\begin{split}
&d_{j} \psi
'(b_{j-1})+(\lambda +\bar{\lambda })\psi
(b_{j-1}) =\frac{
\lambda }{\mu }\, e^{-b_{j-1}/\mu } \int
_{0}^{b_{j-1}} \psi (u)e^{u/
\mu }\,
\mathrm{d}u\nonumber
\\[-2pt]
&\quad +\lambda e^{-b_{j-1}/\mu } +\frac{\bar{\lambda }}{\bar{\mu }} \, e^{b_{j-1}/\bar{\mu }}
\int_{b_{j-1}}^{\infty } \psi (u)e^{-u/\bar{
\mu }}\,
\mathrm{d}u.
%\end{split} %
\end{align}

One more condition is obtained from the equality $\psi (0)=1$. Finally,
the last $2(k-1)$ conditions are derived from the equalities
$d_{j} \psi '_{j}(b_{j})=d_{j+1} \psi '_{j+1}(b_{j})$ and $\psi _{j}(b
_{j})=\psi _{j+1}(b_{j})$ for $1\le j\le k-1$. Note that the first
equality easily follows from \eqref{eq:18}.

Taking into account \eqref{eq:23}, for all $1\le j\le k$, we get:
%
%e31 #&#
\begin{equation}
\label{eq:31} \psi '_{j}(x)=C_{1,j}
z_{1,j} e^{z_{1,j} x} +C_{2,j} z_{2,j}
e^{z_{2,j}
x}, \quad x\in [b_{j-1},b_{j}],\vspace*{-9pt}
\end{equation}
%
%e32 #&#
\begin{align}
\label{eq:32} %
%\begin{split}
&\int_{0}^{b_{j-1}}
\psi (u)e^{u/\mu }\, \mathrm{d}u =\sum_{l=1}^{j-1}
\int_{b_{l-1}}^{b_{l}} \psi _{l}(u)e^{u/\mu }
\, \mathrm{d}u =\sum_{l=1} ^{j-1} \Biggl(
\sum_{i=1}^{2} \frac{C_{i,l}}{z_{i,l}+1/\mu }\nonumber
\\[-2pt]
&\quad \times \bigl( e^{(z_{i,l}+1/\mu )b_{l}} -e^{(z_{i,l}+1/\mu )b
_{l-1}} \bigr)
+C_{3,l} \mu \bigl( e^{b_{l}/\mu } -e^{b_{l-1}/\mu } \bigr)
\Biggr)
%\end{split} %
\end{align}
and
%
%e33 #&#
\begin{align}
\label{eq:33} %
%\begin{split}
&\int_{b_{j-1}}^{\infty }
\psi (u)e^{-u/\bar{\mu }}\, \mathrm{d}u = \sum_{l=j}^{k}
\int_{b_{l-1}}^{b_{l}} \psi _{l}(u)e^{-u/\bar{\mu }}
\, \mathrm{d}u =\sum_{l=j}^{k} \Biggl(
\sum_{i=1}^{2} \frac{C_{i,l}}{z
_{i,l}-1/\bar{\mu }}\nonumber
\\[-2pt]
&\quad \times \bigl( e^{(z_{i,l}-1/\bar{\mu })b_{l}} -e^{(z_{i,l}-1/\bar{
\mu })b_{l-1}} \bigr)
-C_{3,l} \bar{\mu } \bigl( e^{-b_{l}/\bar{\mu }} -e^{-b_{l-1}/\bar{\mu }}
\bigr) \Biggr).
%\end{split} %
\end{align}\newpage

Substituting \eqref{eq:23}, \eqref{eq:31}, \eqref{eq:32} and
\eqref{eq:33} into \eqref{eq:30} and doing some simplifications yield
\eqref{eq:24}. Next, from the equality $\psi (0)=1$ we get
\eqref{eq:25}. Lastly, substituting \eqref{eq:31} into $d_{j} \psi '_{j}(b
_{j})=d_{j+1} \psi '_{j+1}(b_{j})$ and \eqref{eq:23} into $\psi _{j}(b
_{j})=\psi _{j+1}(b_{j})$ for $1\le j\le k-1$ immediately yields
\eqref{eq:26} and \eqref{eq:27}, respectively.\vadjust{\goodbreak}

Thus, we get the system of $3k-1$ linear equations
\eqref{eq:24}--\eqref{eq:27} to determine $3k-1$ unknown constants. That
system has a unique solution provided that its determinant is not equal
to 0. Hence, piecewise differential equation \eqref{eq:19} has a unique
solution satisfying certain conditions and that solution is given by
\eqref{eq:23}. Since we have derived \eqref{eq:19} from~\eqref{eq:18}
without any additional assumptions concerning the differentiability of
$\psi (x)$, we conclude that the functions $\psi _{j}(x)$, $1\le j
\le k$, given by~\eqref{eq:23} are unique solutions to~\eqref{eq:18} on
the intervals $[b_{j-1},b_{j}]$ satisfying certain conditions. This
guaranties that the functions $\psi _{j}(x)$ we have found coincide with
the ruin probability\index{ruin probability} on $[b_{j-1},b_{j}]$, which completes the proof.
\end{proof}

%r6 #&#
\begin{remark}
\label{rem:6}
In particular, if $k=2$, then $C_{3,2}=0$ and the constants
$C_{1,1}$, $C_{2,1}$, $C_{3,1}$, $C_{1,2}$ and $C_{2,2}$ are determined
from the system of linear equations \eqref{eq:34}--\eqref{eq:38}:
%
%e34 #&#
\begin{align}
\label{eq:34} %
%\begin{split}
&\sum_{i=1}^{2}
\biggl( \frac{\bar{\lambda }}{\bar{\mu }z_{i,1}-1}\, \bigl(1- e^{(z_{i,1}-1/\bar{\mu })b_{1}} \bigr)
+d_{1} z_{i,1} +\lambda +\bar{\lambda } \biggr)
C_{i,1}\nonumber
\\
&\quad + \bigl( \bar{\lambda }e^{-b_{1}/\bar{\mu }} +\lambda \bigr) C
_{3,1} +\sum_{i=1}^{2}
\frac{\bar{\lambda }e^{(z_{i,2}-1/\bar{\mu })b
_{1}}}{\bar{\mu }z_{i,2}-1}\, C_{i,2} =\lambda ,\\[-30pt]\nonumber
%\end{split} %
\end{align}
%
%e35 #&#
\begin{align}
\label{eq:35} %
%\begin{split}
&\sum_{i=1}^{2}
\frac{\lambda e^{-b_{1}/\mu }}{\mu z_{i,1}+1}\, \bigl(1- e^{(z_{i,1}+1/\mu )b_{1}} \bigr) C_{i,1} +
\lambda \bigl( e ^{-b_{1}/\mu }-1 \bigr) C_{3,1}\nonumber
\\
&\quad +\sum_{i=1}^{2} \biggl(
\frac{\bar{\lambda }e^{z_{i,2}b_{1}}}{\bar{
\mu }z_{i,2}-1} +(d_{2} z_{i,2} +\lambda +\bar{
\lambda }) e^{z_{i,2}b
_{1}} \biggr) C_{i,2} =\lambda
e^{-b_{1}/\mu },
%\end{split} %
\end{align}
%
%e36 #&#
\begin{equation}
\label{eq:36} C_{1,1} +C_{2,1} +C_{3,1}=1,
\end{equation}
%
%e37 #&#
\begin{equation}
\label{eq:37} d_{1} \sum_{i=1}^{2}
z_{i,1} e^{z_{i,1} b_{1}} C_{i,1} -d_{2}
\sum_{i=1}^{2} z_{i,2}
e^{z_{i,2} b_{1}} C_{i,2}=0
\end{equation}
and
%
%e38 #&#
\begin{equation}
\label{eq:38} \sum_{i=1}^{2}
e^{z_{i,1} b_{1}} C_{i,1} +C_{3,1} -\sum
_{i=1}^{2} e ^{z_{i,2} b_{1}}
C_{i,2}=0
\end{equation}
provided that its determinant is not equal to 0.
\end{remark}

The proposition below enables us to check whether the system of
equations \eqref{eq:34}--\eqref{eq:38} has a unique solution. Let
\begin{align*}
%
%\begin{split}
\Delta &=\frac{1}{\bar{\lambda }(\mu +\bar{\mu })^{2}}
\\
&\times \bigl( d_{1} \bar{\mu }(z_{1,1}
-z_{2,1}) \bigl( e^{b_{1}/\bar{
\mu }} -e^{-b_{1}/\mu } \bigr) \bigl(
(\bar{\lambda }\bar{\mu }- \lambda \mu -d_{1}) e^{b_{1}/\bar{\mu }}
-(\bar{\lambda }\bar{\mu }- \lambda \mu -d_{2}) \bigr)
\\
& \;\;\quad -d_{1} \bar{\mu }e^{b_{1}/\bar{\mu }}
(z_{1,1} -z_{2,1}) (\bar{\lambda }\bar{\mu }-\lambda \mu
-d_{1}) \bigl( e^{b_{1}/\bar{\mu }} -e^{-b
_{1}/\mu } \bigr)
\\
& \;\;\quad +d_{1}\mu (1-1/\bar{\mu }) \bigl( e^{z_{1,1} b_{1}}
-e^{z_{2,1} b_{1}} \bigr) \bigl( (\bar{\lambda }\bar{\mu }-\lambda \mu
-d_{1}) e^{b_{1}/\bar{
\mu }} -(\bar{\lambda }\bar{\mu }-\lambda \mu
-d_{2}) \bigr)
\\
& \;\;\quad +(d_{2}-d_{1}) \bigl( e^{z_{1,1} b_{1}}
-e^{z_{2,1} b_{1}} \bigr) (\bar{ \lambda }\bar{\mu }-\lambda \mu
-d_{1}) \bigl( e^{b_{1}/\bar{\mu }} -e ^{-b_{1}/\mu } \bigr)
\\
& \;\;\quad +d_{1}^{2}\mu e^{b_{1}/\bar{\mu }} (\bar{
\mu }-1) \bigl( z_{2,1} e ^{z_{1,1} b_{1}} -z_{1,1}
e^{z_{2,1} b_{1}} \bigr)
\\
& \;\;\quad +d_{1}\bar{\mu }(d_{2}-d_{1})
\bigl( e^{b_{1}/\bar{\mu }} -e^{-b_{1}/
\mu } \bigr) \bigl( z_{2,1}
e^{z_{1,1} b_{1}} -z_{1,1} e^{z_{2,1} b
_{1}} \bigr) \bigr).
%\end{split}
%
\end{align*}

%p1 #&#
\begin{propos}
\label{pr:1}
The system of linear equations \eqref{eq:34}--\eqref{eq:38} has a unique
solution if and only if $\Delta \neq 0$.
\end{propos}

\begin{proof}
From \eqref{eq:36} we have $C_{3,1} =1 -C_{1,1} -C_{2,1}$. Substituting
that into \eqref{eq:34}, \eqref{eq:35} and \eqref{eq:38} yields
%
%e39 #&#
\begin{align}
\label{eq:39} %
%\begin{split}
&\sum_{i=1}^{2}
\biggl( \frac{\bar{\lambda }}{\bar{\mu }z_{i,1}-1}\, \bigl(1- e^{(z_{i,1}-1/\bar{\mu })b_{1}} \bigr)
+d_{1} z_{i,1} +\bar{ \lambda }-\bar{\lambda
}e^{-b_{1}/\bar{\mu }} \biggr) C_{i,1}\nonumber
\\[-2pt]
&\quad +\sum_{i=1}^{2}
\frac{\bar{\lambda }e^{(z_{i,2}-1/\bar{\mu })b
_{1}}}{\bar{\mu }z_{i,2}-1}\, C_{i,2} =-\bar{\lambda }e^{-b_{1}/\bar{
\mu }},\\[-30pt]\nonumber
%\end{split} %
\end{align}
%
%e40 #&#
\begin{align}
\label{eq:40} %
%\begin{split}
&\sum_{i=1}^{2}
\biggl( \frac{\lambda e^{-b_{1}/\mu }}{\mu z_{i,1}+1} \, \bigl(1- e^{(z_{i,1}+1/\mu )b_{1}} \bigr) +\lambda
\bigl( 1-e^{-b
_{1}/\mu } \bigr) \biggr) C_{i,1}\nonumber
\\[-2pt]
&\quad +\sum_{i=1}^{2} \biggl(
\frac{\bar{\lambda }e^{z_{i,2}b_{1}}}{\bar{
\mu }z_{i,2}-1} +(d_{2} z_{i,2} +\lambda +\bar{
\lambda }) e^{z_{i,2}b
_{1}} \biggr) C_{i,2} =\lambda
%\end{split}
%
\end{align}
and
%
%e41 #&#
\begin{equation}
\label{eq:41} \sum_{i=1}^{2}
\bigl(e^{z_{i,1}b_{1}}-1 \bigr) C_{i,1} -\sum
_{i=1}^{2} e^{z_{i,2}b_{1}}
C_{i,2}=-1.
\end{equation}

Thus, the system of equations \eqref{eq:34}--\eqref{eq:38} has a unique
solution if and only if the system of equations \eqref{eq:39},
\eqref{eq:40}, \eqref{eq:37} and \eqref{eq:41} does.

Multiplying \eqref{eq:41} by $(-d_{2} z_{1,2})$, adding \eqref{eq:37}
and rearranging the terms we obtain
%
%e42 #&#
\begin{equation}
\label{eq:42} e^{z_{2,2} b_{1}} C_{2,2} =\frac{
d_{2} z_{1,2}
-\sum_{i=1}^{2}   ( d_{1} z_{i,1} e^{z_{i,1}b_{1}}
+d
_{2} z_{1,2}   (1-e^{z_{i,1}b_{1}}  )   ) C_{i,1}}{d_{2} (z
_{1,2} -z_{2,2})}.
\end{equation}

Similarly, multiplying \eqref{eq:41} by $(-d_{2} z_{2,2})$, adding
\eqref{eq:37} and rearranging the terms we obtain
%
%e43 #&#
\begin{equation}
\label{eq:43} e^{z_{1,2} b_{1}} C_{1,2} =\frac{
d_{2} z_{2,2}
-\sum_{i=1}^{2}   ( d_{1} z_{i,1} e^{z_{i,1}b_{1}}
+d
_{2} z_{2,2}   (1-e^{z_{i,1}b_{1}}  )   ) C_{i,1}}{d_{2} (z
_{2,2} -z_{1,2})}.
\end{equation}

Substituting \eqref{eq:42} and \eqref{eq:43} into \eqref{eq:39} and
doing some simplifications yield
%
%e44 #&#
\begin{align}
\label{eq:44} %
%\begin{split}
&\sum_{i=1}^{2}
\biggl( \frac{\bar{\mu }z_{i,1} e^{b_{1}/\bar{\mu }} -e
^{z_{i,1}b_{1}}}{\bar{\mu }z_{i,1}-1} +d_{1} z_{i,1}
e^{b_{1}/\bar{
\mu }} -1\nonumber
\\[-2pt]
& \qquad +\frac{-d_{1} \bar{\mu }z_{i,1} e^{z_{i,1}b_{1}}
+d_{2} (1 -\bar{
\mu }z_{1,2} -\bar{\mu }z_{2,2}) (1-e^{z_{i,1}b_{1}})}{d_{2}(\bar{
\mu }z_{1,2}-1)(\bar{\mu }z_{2,2}-1)} \biggr) C_{i,1}\nonumber
\\[-2pt]
&\quad =-\frac{\bar{\mu }^{2} z_{1,2} z_{2,2}}{(\bar{\mu }z_{1,2}-1)(\bar{
\mu }z_{2,2}-1)}.
%\end{split} %
\end{align}

By Vieta's theorem applied to \eqref{eq:29} for $j=2$, we have
\begin{equation*}
z_{1,2} z_{2,2} =\frac{\bar{\lambda }\bar{\mu }-\lambda \mu -d_{2}}{d
_{2} \mu \bar{\mu }},\vadjust{\goodbreak}
\end{equation*}
\begin{equation*}
1 -\bar{\mu }z_{1,2} -\bar{\mu }z_{2,2} =
\frac{d_{2} \bar{\mu }+
\mu \bar{\mu }(\lambda +\bar{\lambda })}{d_{2} \mu }
\end{equation*}
and
\begin{equation*}
(\bar{\mu }z_{1,2}-1) (\bar{\mu }z_{2,2}-1) =
\frac{\bar{\lambda }\bar{
\mu }(\mu +\bar{\mu })}{d_{2} \mu }.
\end{equation*}

Substituting these equalities into \eqref{eq:44} gives
%
%e45 #&#
\begin{align}
\label{eq:45} %
%\begin{split}
&\sum_{i=1}^{2}
\biggl( \frac{\bar{\mu }z_{i,1} e^{b_{1}/\bar{\mu }} -e
^{z_{i,1} b_{1}}}{\bar{\mu }z_{i,1}-1} +d_{1} z_{i,1}
e^{b_{1}/\bar{
\mu }} -1\nonumber
\\[-2pt]
& \qquad +\frac{-d_{1} \mu z_{i,1} e^{z_{i,1}b_{1}}
+(d_{2} +\mu (\lambda +\bar{
\lambda })) (1-e^{z_{i,1}b_{1}})}{\bar{\lambda }(\mu +\bar{\mu })} \biggr) C_{i,1}\nonumber
\\[-2pt]
&\quad =\frac{d_{2} +\lambda \mu -\bar{\lambda }\bar{\mu }}{\bar{
\lambda }(\mu +\bar{\mu })}.
%\end{split} %
\end{align}

Next, multiplying \eqref{eq:41} by $(\lambda +\bar{\lambda })$ and
adding \eqref{eq:37} and \eqref{eq:40} we get
%
%e46 #&#
\begin{align}
\label{eq:46} %
%\begin{split}
&\sum_{i=1}^{2}
\biggl( \frac{\lambda e^{-b_{1}/\mu }}{\mu z_{i,1}+1} \, \bigl(1- e^{(z_{i,1}+1/\mu )b_{1}} \bigr) +\lambda
\bigl( 1-e^{-b
_{1}/\mu } \bigr) +d_{1} z_{i,1}
e^{z_{i,1}b_{1}}\nonumber
\\[-2pt]
&\quad +(\lambda +\bar{\lambda }) \bigl( e^{z_{i,1}b_{1}}-1 \bigr) \biggr)
C_{i,1} +\sum_{i=1}^{2}
\frac{\bar{\lambda }e^{z_{i,2}b_{1}}}{\bar{
\mu }z_{i,2}-1} C_{i,2} =-\bar{\lambda }.
%\end{split}
%
\end{align}

Multiplying \eqref{eq:39} by $(-e^{b_{1}/\bar{\mu }})$ and adding
\eqref{eq:46} we obtain
%
%e47 #&#
\begin{align}
\label{eq:47} %
%\begin{split}
&\sum_{i=1}^{2}
\biggl( \frac{-\lambda \mu z_{i,1} e^{-b_{1}/\mu } -
\lambda e^{z_{i,1}b_{1}}}{\mu z_{i,1}+1} +\frac{-\bar{\lambda }e^{b
_{1}/\bar{\mu }} +\bar{\lambda }e^{z_{i,1}b_{1}}}{\bar{\mu }z_{i,1}-1}\nonumber
\\[-2pt]
&\quad -(d_{1} z_{i,1} +\bar{\lambda })
e^{b_{1}/\bar{\mu }} +(d_{1} z_{i,1} +\lambda +\bar{
\lambda }) e^{z_{i,1}b_{1}} \biggr) C_{i,1} = \lambda -\bar{\lambda
}.
%\end{split} %
\end{align}

Thus, if the system of equations \eqref{eq:45} and \eqref{eq:47} has a
unique solution, then $C_{1,2}$ and $C_{2,2}$ can be found from
\eqref{eq:43} and \eqref{eq:42}, respectively. Consequently, the system
of equations \eqref{eq:39}, \eqref{eq:40}, \eqref{eq:37} and
\eqref{eq:41} has a unique solution if and only if the system of
equations \eqref{eq:45} and \eqref{eq:47} does.

A standard computation shows that the determinant of the system of
equations \eqref{eq:45} and \eqref{eq:47} is equal to $\Delta $ defined
above. In particular, here we use Vieta's theorem applied to
\eqref{eq:29} for $j=1$. Therefore, the system of equations
\eqref{eq:34}--\eqref{eq:38} has a unique solution if and only if
$\Delta \neq 0$.
\end{proof}\newpage

%s5.2 #&#
\subsection{Explicit formulas for the expected discounted dividend
payments\index{dividend ! payments} until ruin}%
\label{sec:5.2}
By \eqref{eq:17}, equation \eqref{eq:9} for the expected discounted
dividend payments\index{dividend ! payments} can be written as
%
%e48 #&#
\begin{align}
\label{eq:48} %
%\begin{split}
&d_{j}
v'(x)+(\lambda +\bar{\lambda }+\delta )v(x)\nonumber
\\[-2pt]
&\quad =\frac{\lambda }{\mu }\, e^{-x/\mu } \int_{0}^{x}
v(u)e^{u/
\mu }\, \mathrm{d}u +\frac{\bar{\lambda }}{\bar{\mu }}\, e^{x/\bar{
\mu }}
\int_{x}^{\infty } v(u)e^{-u/\bar{\mu }}\,
\mathrm{d}u +d_{j}
%\end{split} %
\end{align}
for all $x\in [b_{j-1},b_{j}]$ and $1\le j\le k$.

The piecewise integro-differential equation \eqref{eq:48} can also be
reduced to a piecewise linear differential equation with constant
coefficients.

%l2 #&#
\begin{lemma}
\label{lem:2}Let the surplus process\index{surplus process} $  (X_{t}^{\mathbf{b}}(x)  )_{t\ge 0}$
%follow
be defined by \eqref{eq:2} under the above assumptions, and let claim and
premium sizes be exponentially distributed with means $\mu $ and
$\bar{\mu }$, respectively. Then $v(x)$ is a solution to the piecewise
differential equation
%
%e49 #&#
\begin{align}
\label{eq:49} %
%\begin{split}
&d_{j}\mu \bar{\mu
}v'''(x) + \bigl(
d_{j}(\bar{\mu }-\mu ) +\mu \bar{ \mu }(\lambda +\bar{\lambda }+
\delta ) \bigr)v''(x)\nonumber
\\[-2pt]
&\quad +\bigl(\bar{\mu }(\bar{\lambda }+\delta ) -\mu (\lambda +\delta ) -d
_{j}\bigr)v'(x) -\delta v(x)=-d_{j}
%\end{split} %
\end{align}
for all $x\in [b_{j-1},b_{j}]$ and $1\le j\le k$.
\end{lemma}

The proof of the lemma is similar to the proof of Lemma~\ref{lem:1}.

For $1\le j\le k$, let
\begin{equation*}
\begin{split}
\tilde{\mathrm{D}}_{j} &=-18\delta
d_{j} \mu \bar{\mu } \bigl( d_{j}(\bar{ \mu }-\mu ) +
\mu \bar{\mu }(\lambda +\bar{\lambda }+\delta ) \bigr) \bigl( \bar{\mu }(\bar{
\lambda }+\delta ) -\mu (\lambda +\delta ) -d _{j} \bigr)
\\
&\quad +4\delta \bigl( d_{j}(\bar{\mu }-\mu ) +\mu \bar{\mu }(
\lambda +\bar{\lambda }+\delta ) \bigr)^{3}
\\
&\quad + \bigl( d_{j}(\bar{\mu }-\mu ) +\mu \bar{\mu }(\lambda +
\bar{ \lambda }+\delta ) \bigr)^{2} \bigl( \bar{\mu }(\bar{\lambda }+
\delta ) -\mu (\lambda +\delta ) -d_{j} \bigr)^{2}
\\
&\quad -4d_{j} \mu \bar{\mu } \bigl( \bar{\mu }(\bar{\lambda }+
\delta ) -\mu (\lambda +\delta ) -d_{j} \bigr)^{3} -27
(\delta d_{j} \mu \bar{ \mu })^{2}. \end{split}
\end{equation*}

%t4 #&#
\begin{thm}
\label{thm:4}Let the surplus process\index{surplus process} $  (X_{t}^{\mathbf{b}}(x)  )_{t\ge 0}$
%follow
be defined by \eqref{eq:2} under the above assumptions, and let claim and
premium sizes be exponentially distributed with means $\mu $ and
$\bar{\mu }$, respectively. If the net profit condition\index{net profit condition} \eqref{eq:3}
holds and $\min _{1\le j\le k} \tilde{\mathrm{D}}_{j}>0$, then we
have
%
%e50 #&#
\begin{equation}
\label{eq:50} v_{j}(x)=\tilde{C}_{1,j}
e^{\tilde{z}_{1,j} x} +\tilde{C}_{2,j} e^{
\tilde{z}_{2,j} x} +
\tilde{C}_{3,j} e^{\tilde{z}_{3,j} x} +d_{j}/ \delta
\end{equation}
for all $x\in [b_{j-1},b_{j}]$ and $1\le j\le k$, where $\tilde{z}
_{1,j}$, $\tilde{z}_{2,j}$ and $\tilde{z}_{3,j}$ are distinct real roots
of the cubic equation
%
%e51 #&#
\begin{equation}
\label{eq:51} d_{j} \mu \bar{\mu }z^{3} + \bigl(
d_{j} (\bar{\mu }-\mu ) +\mu \bar{ \mu }(\lambda +\bar{\lambda }+
\delta ) \bigr) z^{2} + \bigl( \bar{ \mu }(\bar{\lambda }+\delta ) -
\mu (\lambda +\delta ) -d_{j} \bigr) z -\delta =0,
\end{equation}
$\tilde{C}_{3,k}=0$ and all the other constants $\tilde{C}_{1,j}$,
$\tilde{C}_{2,j}$ and $\tilde{C}_{3,j}$ are determined from the system
of linear equations \eqref{eq:52}--\eqref{eq:55}:
%
%e52 #&#
\begin{align}
\label{eq:52} %
%\begin{split}
&\lambda e^{-b_{j-1}/\mu } \sum
_{l=1}^{j-1} \Biggl( \sum
_{i=1}^{3} \frac{
\tilde{C}_{i,l}}{\mu \tilde{z}_{i,l}+1}\,
\bigl(e^{(\tilde{z}_{i,l}+1/
\mu )b_{l}} -e^{(\tilde{z}_{i,l}+1/\mu )b_{l-1}} \bigr) \Biggr)\nonumber
\\[-2pt]
&\quad +\sum_{i=1}^{3} \biggl(
\frac{\bar{\lambda }e^{b_{j-1}/\bar{
\mu }}}{\bar{\mu }\tilde{z}_{i,j}-1}\, \bigl(e^{(\tilde{z}_{i,j}-1/\bar{
\mu })b_{j}} - e^{(\tilde{z}_{i,j}-1/\bar{\mu })b_{j-1}} \bigr)\nonumber
\\[-2pt]
& \qquad -(d_{j} \tilde{z}_{i,j} +\lambda +\bar{
\lambda }+\delta ) e^{
\tilde{z}_{i,j}b_{j-1}} \biggr) \tilde{C}_{i,j}\nonumber
\\[-2pt]
&\quad +\bar{\lambda }e^{b_{j-1}/\bar{\mu }} \sum_{l=j+1}^{k}
\Biggl( \sum_{i=1}^{3}
\frac{\tilde{C}_{i,l}}{\bar{\mu }\tilde{z}_{i,l}-1} \, \bigl(e^{(\tilde{z}_{i,l}-1/\bar{\mu })b_{l}} -e^{(\tilde{z}_{i,l}-1/\bar{
\mu })b_{l-1}} \bigr)
\Biggr)\nonumber
\\[-2pt]
&\quad =\frac{d_{j}(\lambda +\bar{\lambda })}{\delta } -\frac{\lambda
e^{-b_{j-1}/\mu }}{\delta } \sum
_{l=1}^{j-1} d_{l}
\bigl(e^{b_{l}/
\mu } -e^{b_{l-1}/\mu } \bigr)\nonumber
\\[-2pt]
& \qquad +\frac{\bar{\lambda }e^{b_{j-1}/\bar{\mu }}}{\delta } \sum_{l=j}^{k}
d _{l} \bigl(e^{-b_{l}/\bar{\mu }} -e^{-b_{l-1}/\bar{\mu }} \bigr), \quad 1
\le j\le k,
%\end{split} %
\end{align}
%
%e53 #&#
\begin{equation}
\label{eq:53} \tilde{C}_{1,1} +\tilde{C}_{2,1} +
\tilde{C}_{3,1}=-d_{1}/\delta ,
\end{equation}
%
%e54 #&#
\begin{align}
\label{eq:54} %
%\begin{split}
&d_{j} \sum
_{i=1}^{3} \tilde{z}_{i,j}
e^{\tilde{z}_{i,j} b_{j}} \tilde{C}_{i,j} -d_{j+1} \sum
_{i=1}^{3} \tilde{z}_{i,j+1}
e^{
\tilde{z}_{i,j+1} b_{j}} \tilde{C}_{i,j+1}\nonumber
\\[-2pt]
&\quad =d_{j} -d_{j+1}, \quad 1\le j\le k-1,
%\end{split} %
\end{align}
and
%
%e55 #&#
\begin{equation}
\label{eq:55} \sum_{i=1}^{3}
e^{\tilde{z}_{i,j} b_{j}} \tilde{C}_{i,j} -\sum
_{i=1} ^{3} e^{\tilde{z}_{i,j+1} b_{j}}
\tilde{C}_{i,j+1} =\frac{d_{j+1}-d
_{j}}{\delta }, \quad 1\le j\le k-1,
\end{equation}
provided that its determinant is not equal to 0.
\end{thm}

\begin{proof}
By Lemma~\ref{lem:2} and the notation introduced in Section
\ref{sec:2}, we deduce that the function $v_{j}(x)$ is a solution to
\eqref{eq:49} on $x\in [b_{j-1},b_{j}]$ for each $1\le j\le k$. In
addition, it is easily seen that \eqref{eq:51} is the corresponding
characteristic equation and its discriminant coincides with the constant
$\tilde{\mathrm{D}}_{j}$ introduced above. The assumption $\min _{1\le j\le k} \tilde{\mathrm{D}}_{j}>0$ guarantees that cubic
equation \eqref{eq:51} has three distinct real roots $\tilde{z}_{1,j}$,
$\tilde{z}_{2,j}$ and $\tilde{z}_{3,j}$. Hence, the general solution to
\eqref{eq:49} is given by \eqref{eq:50} with some constants
$\tilde{C}_{1,j}$, $\tilde{C}_{2,j}$ and $\tilde{C}_{3,j}$.

By Vieta's theorem, we conclude that \eqref{eq:51} has either two or no
negative roots for each $1\le j\le k$. Since the net profit condition\index{net profit condition}
\eqref{eq:3} holds, applying arguments similar to those in
\cite[p.~70]{Sc2008} shows that $\lim_{x\to \infty } v_{k}(x) =d_{k}/
{\delta }$. \xch{Consequently}{Concequently}, if equation \eqref{eq:51} for $j=k$ had no
negative roots, the function $v_{k}(x)$ would be constant, which is
impossible. Therefore, equation \eqref{eq:51} for $j=k$ has two negative roots.
We denote those negative roots by $\tilde{z}_{1,k}$ and $\tilde{z}
_{2,k}$, and let $\tilde{z}_{3,k}$ be the third root. Since
$\tilde{z}_{3,k}>0$, we conclude that $\tilde{C}_{3,k}=0$.

To determine all the other constants $\tilde{C}_{1,j}$, $\tilde{C}
_{2,j}$ and $\tilde{C}_{3,j}$, we need $3k-1$ boundary conditions. The
first $k$ conditions are obtained by letting $x=b_{j-1}$ in
\eqref{eq:48} for $1\le j\le k$:
%
%e56 #&#
\begin{align}
\label{eq:56} %
%\begin{split}
&d_{j}
v'(b_{j-1})+(\lambda +\bar{\lambda }+\delta
)v(b_{j-1})\nonumber
\\
&\quad =\frac{\lambda }{\mu }\, e^{-b_{j-1}/\mu } \int_{0}^{b_{j-1}}
v(u)e^{u/\mu }\, \mathrm{d}u +\frac{\bar{\lambda }}{\bar{\mu }}\, e ^{b_{j-1}/\bar{\mu }}
\int_{b_{j-1}}^{\infty } v(u)e^{-u/\bar{\mu }} \,
\mathrm{d}u +d_{j}.
%\end{split} %
\end{align}\newpage

One more condition is obtained from the equality $v(0)=0$. The last
$2(k-1)$ conditions are derived from the equalities $d_{j} v'_{j}(b
_{j}) -d_{j} =d_{j+1} v'_{j+1}(b_{j}) -d_{j+1}$, which easily %follows
follow
from \eqref{eq:48}, and $v_{j}(b_{j})=v_{j+1}(b_{j})$ for $1\le j
\le k-1$.

Substituting \eqref{eq:50} into \eqref{eq:56} \xch{as well as into the}{,
applying} equalities $v(0)=0$,
$d_{j} v'_{j}(b_{j}) -d_{j} =d_{j+1} v'_{j+1}(b_{j}) -d_{j+1}$ and
$v_{j}(b_{j}) =v_{j+1}(b_{j})$ for $1\le j\le k-1$ and doing some
simplifications yield the system of linear equations
\eqref{eq:52}--\eqref{eq:55}, which has a unique solution provided that
its determinant is not equal to 0. Thus, piecewise differential equation
\eqref{eq:49} has a unique solution satisfying certain conditions, and
that solution is given by \eqref{eq:50}. Applying arguments similar to
those in the proof of Theorem~\ref{thm:3} guaranties that the functions
$v_{j}(x)$ we have found coincide with the expected discounted dividend
payments\index{dividend ! payments} until ruin on $[b_{j-1},b_{j}]$, which completes the proof.
\end{proof}

%r7 #&#
\begin{remark}
\label{rem:7}
In particular, if $k=2$, then $\tilde{C}_{3,2}=0$ and the constants
$\tilde{C}_{1,1}$, $\tilde{C}_{2,1}$, $\tilde{C}_{3,1}$, $\tilde{C}
_{1,2}$ and $\tilde{C}_{2,2}$ are determined from the system of linear
equations \eqref{eq:57}--\eqref{eq:61}:
%
%e57 #&#
\begin{align}
\label{eq:57} %
%\begin{split}
&\sum_{i=1}^{3}
\biggl( \frac{\bar{\lambda }}{\bar{\mu }\tilde{z}_{i,1}-1}\, \bigl(1-e^{(
\tilde{z}_{i,1}-1/\bar{\mu })b_{1}} \bigr)
+d_{1} \tilde{z}_{i,1} + \lambda +\bar{\lambda }+\delta
\biggr) \tilde{C}_{i,1}\nonumber
\\
&\quad +\sum_{i=1}^{2}
\frac{\bar{\lambda }e^{(\tilde{z}_{i,2}-1/\bar{
\mu })b_{1}}}{\bar{\mu }\tilde{z}_{i,2}-1}\, \tilde{C}_{i,2} =-\frac{d
_{1} \lambda }{\delta } -
\frac{\bar{\lambda }(d_{1}-d_{2}) e^{-b_{1}/\bar{
\mu }}}{\delta },\\[-24pt]\nonumber
%\end{split} %
\end{align}
%
%e58 #&#
\begin{align}
\label{eq:58} %
%\begin{split}
&\lambda e^{-b_{1}/\mu } \sum
_{i=1}^{3} \frac{\tilde{C}_{i,1}}{\mu
\tilde{z}_{i,1}+1}\,
\bigl(1-e^{(\tilde{z}_{i,1}+1/\mu )b_{1}} \bigr) + \sum_{i=1}^{2}
\biggl( \frac{\bar{\lambda }e^{\tilde{z}_{i,2} b_{1}}}{\bar{
\mu }\tilde{z}_{i,2}-1}\nonumber
\\
&\quad +(d_{2} \tilde{z}_{i,2} +\lambda +\bar{\lambda
}+\delta ) e ^{\tilde{z}_{i,2}b_{1}} \biggr) \tilde{C}_{i,2} =
\frac{\lambda (d_{1}-d
_{2})}{\delta } -\frac{d_{1} \lambda e^{-b_{1}/\mu }}{\delta },
%\end{split} %
\end{align}
%
%e59 #&#
\begin{equation}
\label{eq:59} \tilde{C}_{1,1} +\tilde{C}_{2,1} +
\tilde{C}_{3,1}=-d_{1}/\delta ,
\end{equation}
%
%e60 #&#
\begin{equation}
\label{eq:60} d_{1} \sum_{i=1}^{3}
\tilde{z}_{i,1} e^{\tilde{z}_{i,1} b_{1}} \tilde{C}_{i,1}
-d_{2} \sum_{i=1}^{2}
\tilde{z}_{i,2} e^{\tilde{z}
_{i,2} b_{1}} \tilde{C}_{i,2}=d_{1}
-d_{2}
\end{equation}
and
%
%e61 #&#
\begin{equation}
\label{eq:61} \sum_{i=1}^{3}
e^{\tilde{z}_{i,1} b_{1}} \tilde{C}_{i,1} -\sum
_{i=1} ^{2} e^{\tilde{z}_{i,2} b_{1}}
\tilde{C}_{i,2} =\frac{d_{2}-d_{1}}{
\delta }
\end{equation}
provided that its determinant is not equal to 0.
\end{remark}

%s6 #&#
\section{Numerical illustrations}%
\label{sec:6}
To present numerical examples for the results obtained in
Section~\ref{sec:5}, we set $\lambda =0.1$, $\bar{\lambda }=2.3$,
$\mu =3$, $\bar{\mu }=0.2$, $b=5$ and $\delta =0.01$.

In addition, we denote by $\psi ^{*}(x)$ the ruin probability\index{ruin probability} in the
corresponding model without dividend payments.\index{dividend ! payments} It is given by
\begin{equation*}
\psi ^{*}(x)=\frac{\lambda (\mu +\bar{\mu })}{\bar{\mu }(\lambda +\bar{
\lambda })}\, \exp \biggl( -
\frac{(\bar{\lambda }\bar{\mu }-\lambda
\mu )x}{\mu \bar{\mu }(\lambda +\bar{\lambda })} \biggr), \quad x \in [0,\infty ),
\end{equation*}
(see \cite{Boi2003,MiRa2016}). For the parameters chosen above,
$\psi ^{*}(x) \approx 0.666667 e^{-0.111111 x}$.\vadjust{\goodbreak}

Moreover, let now $d_{1}=0.05$ and $d_{2}=0.1$. Applying
Theorems~\ref{thm:3} and~\ref{thm:4} as well as Remarks~\ref{rem:6}
and~\ref{rem:7} we can calculate the ruin probability $\psi (x)$\index{ruin probability} and the
expected discounted dividend payments\index{dividend ! payments} until ruin $v(x)$:
\begin{gather*}
\psi _{1}(x) \approx 0.530821 e^{-0.084781 x} +0.179668
e^{-43.248552 x} +0.289512, \quad x\in [0,5],
\\
\psi _{2}(x) \approx 0.826718 e^{-0.051863 x} -7.043723 \cdot
10^{38} e ^{-19.28147 x}, \quad x\in [5,\infty );
\\
\begin{split}
v_{1}(x) \approx 5 &- 2.992137
e^{-43.470279 x} -4.421273 e^{-0.124597
x}
\\
&+2.41341 e^{0.061543 x}, \quad x\in [0,5],
\end{split} %
\\
v_{2}(x) \approx 10 -2.198169 \cdot 10^{40}
e^{-19.405407 x} -6.97712 e^{-0.107684 x}, \quad x\in [5,\infty ).
\end{gather*}

Table~\ref{table:1} presents the results of calculations for some values
of $x$.

%t1 #&#
\begin{table}
\caption{The ruin probabilities without and with dividend payments and the expected discounted
dividend payments, $d_{1}=0.05$ and $d_{2}=0.1$}%
\label{table:1}
\begin{tabular*}{6cm}{@{\extracolsep{\fill}}clll@{}}
\hline
$x$ & \multicolumn{1}{c}{$\psi ^{*}(x)$} & \multicolumn{1}{c}{$\psi (x)$} & \multicolumn{1}{c}{$v(x)$} \\
\hline
0 &0.666667 &1 &0\\
1 &0.596560 &0.777184 &3.663273\\
2 &0.533825 &0.737542 &4.283457\\
5 &0.382502 &0.636926 &5.911685\\
7 &0.306284 &0.575029 &6.716708\\
10 &0.219462 &0.492173 &7.623108\\
15 &0.125917 &0.379750 &8.612682\\
20 &0.072245 &0.293007 &9.190265\\
50 &0.002577 &0.061825 &9.967986\\
70 &0.000279 &0.021912 &9.996285\\
\hline
\end{tabular*}
\end{table}

Next, for $d_{1}=0.1$ and $d_{2}=0.05$, we get
\begin{gather*}
\psi _{1}(x) \approx 1.204304 e^{-0.051863 x} +0.218067
e^{-19.28147 x} -0.42237, \quad x\in [0,5],
\\
\psi _{2}(x) \approx 0.772527 e^{-0.084781 x} +1.012903 \cdot
10^{91} e ^{-43.248552 x}, \quad x\in [5,\infty );
\\
\begin{split}
v_{1}(x) \approx 10 &- 2.219094
e^{-19.405407 x} -5.609737 e^{-0.107684
x}
\\[-2pt]
&-2.171169 e^{0.079758 x}, \quad x\in [0,5],
\end{split} %
\\
v_{2}(x) \approx 5 +5.716149 \cdot 10^{92}
e^{-43.470279 x} -2.857069 e^{-0.124597 x}, \quad x\in [5,\infty ).
\end{gather*}

The values of $\psi ^{*}(x)$, $\psi (x)$ and $v(x)$ for some $x$ are
given in Table~\ref{table:2}.

%t2 #&#
\begin{table}
\caption{The ruin probabilities without and with dividend payments and the expected discounted
dividend payments, $d_{1}=0.1$ and $d_{2}=0.05$}%
\label{table:2}
\begin{tabular*}{6cm}{@{\extracolsep{\fill}}clll@{}}
\hline
$x$ & \multicolumn{1}{c}{$\psi ^{*}(x)$} & \multicolumn{1}{c}{$\psi (x)$} & \multicolumn{1}{c}{$v(x)$} \\
\hline
0 &0.666667 &1 &0\\
1 &0.596560 &0.721066 &2.611525\\
2 &0.533825 &0.663275 &2.930525\\
5 &0.382502 &0.506845 &3.490686\\
7 &0.306284 &0.426750 &3.805635\\
10 &0.219462 &0.330912 &4.178134\\
15 &0.125917 &0.216577 &4.559200\\
20 &0.072245 &0.141747 &4.763582\\
50 &0.002577 &0.011141 &4.994372\\
70 &0.000279 &0.002044 &4.999534\\
\hline
\end{tabular*}
\end{table}

The results presented in Tables~\ref{table:1} and~\ref{table:2} indicate
that dividend payments\index{dividend ! payments} substantially increase the ruin probability.\index{ruin probability} The
first strategy is much more profitable, although the corresponding ruin
probability\index{ruin probability} is larger in that case.

%\begin{appendix}
%\end{appendix}

\begin{acknowledgement}%[title={Acknowledgments}]
The author is deeply grateful to the anonymous referees for careful
reading and valuable comments and suggestions, which helped to improve
the earlier version of the paper.
\end{acknowledgement}

%\begin{funding}
%\gsponsor[id=,sponsor-id=]{}
%\gnumber[refid=]{}
%\end{funding}


\begin{thebibliography}{99}

%b1 ###
\bibitem{AlHa2007}
\begin{barticle}
\bauthor{\bsnm{Albrecher}, \binits{H.}},
\bauthor{\bsnm{Hartinger}, \binits{J.}}:
\batitle{A risk model with multilayer dividend strategy}.
\bjtitle{N. Am. Actuar. J.}
\bvolume{11},
\bfpage{43}--\blpage{64}
(\byear{2007}).
\bid{doi={10.1080/10920277.2007.\\10597447}, mr={2380719}}
\end{barticle}
%
\OrigBibText
\begin{barticle}
\bauthor{\bsnm{Albrecher}, \binits{H.}},
\bauthor{\bsnm{Hartinger}, \binits{J.}}:
\batitle{A risk model with multilayer dividend strategy}.
\bjtitle{North American Actuarial Journal}
\bvolume{11},
\bfpage{43}--\blpage{64}
(\byear{2007})
\end{barticle}
\endOrigBibText
\bptok{structpyb}%
\endbibitem

%b2 ###
\bibitem{AsAl2010}
\begin{bbook}
\bauthor{\bsnm{Asmussen}, \binits{S.}},
\bauthor{\bsnm{Albrecher}, \binits{H.}}:
\bbtitle{Ruin Probabilities}.
\bpublisher{World Scientific},
\blocation{Singapore}
(\byear{2010}).
\bid{doi={10.1142/9789814282536}, mr={2766220}}
\end{bbook}
%
\OrigBibText
\begin{bbook}
\bauthor{\bsnm{Asmussen}, \binits{S.}},
\bauthor{\bsnm{Albrecher}, \binits{H.}}:
\bbtitle{Ruin Probabilities}.
\bpublisher{World Scientific},
\blocation{Singapore}
(\byear{2010})
\end{bbook}
\endOrigBibText
\bptok{structpyb}%
\endbibitem

%b3 ###
\bibitem{BaLa2008}
\begin{barticle}
\bauthor{\bsnm{Badescu}, \binits{A.}},
\bauthor{\bsnm{Landriault}, \binits{D.}}:
\batitle{Recursive calculation of the dividend moments in a multi-threshold
 risk model}.
\bjtitle{N. Am. Actuar. J.}
\bvolume{12},
\bfpage{74}--\blpage{88}
(\byear{2008}).
\bid{doi={\\10.1080/10920277.2008.10597501}, mr={2485710}}
\end{barticle}
%
\OrigBibText
\begin{barticle}
\bauthor{\bsnm{Badescu}, \binits{A.}},
\bauthor{\bsnm{Landriault}, \binits{D.}}:
\batitle{Recursive calculation of the dividend moments in a multi-threshold
 risk model}.
\bjtitle{North American Actuarial Journal}
\bvolume{12},
\bfpage{74}--\blpage{88}
(\byear{2008})
\end{barticle}
\endOrigBibText
\bptok{structpyb}%
\endbibitem

%b4 ###
\bibitem{BaDrLa2007}
\begin{barticle}
\bauthor{\bsnm{Badescu}, \binits{A.}},
\bauthor{\bsnm{Drekic}, \binits{S.}},
\bauthor{\bsnm{Landriault}, \binits{D.}}:
\batitle{On the analysis of a multi-threshold \xch{Markovian}{markovian} risk model}.
\bjtitle{Scand. Actuar. J.}
\bvolume{2007},
\bfpage{248}--\blpage{260}
(\byear{2007}).
\bid{doi={\\10.1080/03461230701554080}, mr={2416548}}
\end{barticle}
%
\OrigBibText
\begin{barticle}
\bauthor{\bsnm{Badescu}, \binits{A.}},
\bauthor{\bsnm{Drekic}, \binits{S.}},
\bauthor{\bsnm{Landriault}, \binits{D.}}:
\batitle{On the analysis of a multi-threshold markovian risk model}.
\bjtitle{Scandinavian Actuarial Journal}
\bvolume{2007},
\bfpage{248}--\blpage{260}
(\byear{2007})
\end{barticle}
\endOrigBibText
\bptok{structpyb}%
\endbibitem

%b5 ###
\bibitem{Boi2003}
\begin{barticle}
\bauthor{\bsnm{Boikov}, \binits{A.V.}}:
\batitle{The {C}ram{\'e}r--{L}undberg model with stochastic premium process}.
\bjtitle{Theory Probab. Appl.}
\bvolume{47},
\bfpage{489}--\blpage{493}
(\byear{2003}).
\bid{doi={10.1137/\\S0040585X9797987}, mr={1975908}}
\end{barticle}
%
\OrigBibText
\begin{barticle}
\bauthor{\bsnm{Boikov}, \binits{A.V.}}:
\batitle{The {C}ram{\'e}r--{L}undberg model with stochastic premium process}.
\bjtitle{Theory of Probability and Its Applications}
\bvolume{47},
\bfpage{489}--\blpage{493}
(\byear{2003})
\end{barticle}
\endOrigBibText
\bptok{structpyb}%
\endbibitem

%b6 ###
\bibitem{BoCoLaMa2006}
\begin{barticle}
\bauthor{\bsnm{Boudreault}, \binits{M.}},
\bauthor{\bsnm{Cossette}, \binits{H.}},
\bauthor{\bsnm{Landriault}, \binits{D.}},
\bauthor{\bsnm{Marceau}, \binits{E.}}:
\batitle{On a risk model with dependence between interclaim arrivals and claim
 sizes}.
\bjtitle{Scand. Actuar. J.}
\bvolume{2006},
\bfpage{265}--\blpage{285}
(\byear{2006}).
\bid{doi={10.1080/03461230600992266}, mr={2328676}}
\end{barticle}
%
\OrigBibText
\begin{barticle}
\bauthor{\bsnm{Boudreault}, \binits{M.}},
\bauthor{\bsnm{Cossette}, \binits{H.}},
\bauthor{\bsnm{Landriault}, \binits{D.}},
\bauthor{\bsnm{Marceau}, \binits{E.}}:
\batitle{On a risk model with dependence between interclaim arrivals and claim
 sizes}.
\bjtitle{Scandinavian Actuarial Journal}
\bvolume{2006},
\bfpage{265}--\blpage{285}
(\byear{2006})
\end{barticle}
\endOrigBibText
\bptok{structpyb}%
\endbibitem

%b7 ###
\bibitem{ChVr2014}
\begin{barticle}
\bauthor{\bsnm{Chadjiconstantinidis}, \binits{S.}},
\bauthor{\bsnm{Vrontos}, \binits{S.}}:
\batitle{On a renewal risk process with dependence under a
 {F}arlie--{G}umbel--{M}orgenstern copula}.
\bjtitle{Scand. Actuar. J.}
\bvolume{2014},
\bfpage{125}--\blpage{158}
(\byear{2014}).
\bid{doi={10.1080/03461238.2012.663730}, mr={3177095}}
\end{barticle}
%
\OrigBibText
\begin{barticle}
\bauthor{\bsnm{Chadjiconstantinidis}, \binits{S.}},
\bauthor{\bsnm{Vrontos}, \binits{S.}}:
\batitle{On a renewal risk process with dependence under a
 {F}arlie--{G}umbel--{M}orgenstern copula}.
\bjtitle{Scandinavian Actuarial Journal}
\bvolume{2014},
\bfpage{125}--\blpage{158}
(\byear{2014})
\end{barticle}
\endOrigBibText
\bptok{structpyb}%
\endbibitem

%b8 ###
\bibitem{ChTa2003}
\begin{barticle}
\bauthor{\bsnm{Cheng}, \binits{Y.}},
\bauthor{\bsnm{Tang}, \binits{Q.}}:
\batitle{Moments of the surplus before ruin and the deficit of ruin in the
 {E}rlang(2) risk process}.
\bjtitle{N. Am. Actuar. J.}
\bvolume{7},
\bfpage{1}--\blpage{12}
(\byear{2003}).
\bid{doi={\\10.1080/10920277.2003.10596073}, mr={1986883}}
\end{barticle}
%
\OrigBibText
\begin{barticle}
\bauthor{\bsnm{Cheng}, \binits{Y.}},
\bauthor{\bsnm{Tang}, \binits{Q.}}:
\batitle{Moments of the surplus before ruin and the deficit of ruin in the
 {E}rlang(2) risk process}.
\bjtitle{North American Actuarial Journal}
\bvolume{7},
\bfpage{1}--\blpage{12}
(\byear{2003})
\end{barticle}
\endOrigBibText
\bptok{structpyb}%
\endbibitem

%b9 ###
\bibitem{ChLi2011}
\begin{barticle}
\bauthor{\bsnm{Chi}, \binits{Y.}},
\bauthor{\bsnm{Lin}, \binits{X.S.}}:
\batitle{On the threshold dividend strategy for a generalized jump-diffusion
 risk model}.
\bjtitle{Insur. Math. Econ.}
\bvolume{48},
\bfpage{326}--\blpage{337}
(\byear{2011}).
\bid{doi={\\10.1016/j.insmatheco.2010.11.006}, mr={2820045}}
\end{barticle}
%
\OrigBibText
\begin{barticle}
\bauthor{\bsnm{Chi}, \binits{Y.}},
\bauthor{\bsnm{Lin}, \binits{X.S.}}:
\batitle{On the threshold dividend strategy for a generalized jump-diffusion
 risk model}.
\bjtitle{Insurance: Mathematics and Economics}
\bvolume{48},
\bfpage{326}--\blpage{337}
(\byear{2011})
\end{barticle}
\endOrigBibText
\bptok{structpyb}%
\endbibitem

%b10 ###
\bibitem{CoMaMa2008}
\begin{barticle}
\bauthor{\bsnm{Cossette}, \binits{H.}},
\bauthor{\bsnm{Marceau}, \binits{E.}},
\bauthor{\bsnm{Marri}, \binits{F.}}:
\batitle{On the compound {P}oisson risk model with dependence based on a
 generalized {F}arlie--{G}umbel--{M}orgenstern copula}.
\bjtitle{Insur. Math. Econ.}
\bvolume{43},
\bfpage{444}--\blpage{455}
(\byear{2008}).
\bid{doi={10.1016/j.insmatheco.2008.08.009}, mr={2479592}}
\end{barticle}
%
\OrigBibText
\begin{barticle}
\bauthor{\bsnm{Cossette}, \binits{H.}},
\bauthor{\bsnm{Marceau}, \binits{E.}},
\bauthor{\bsnm{Marri}, \binits{F.}}:
\batitle{On the compound {P}oisson risk model with dependence based on a
 generalized {F}arlie--{G}umbel--{M}orgenstern copula}.
\bjtitle{Insurance: Mathematics and Economics}
\bvolume{43},
\bfpage{444}--\blpage{455}
(\byear{2008})
\end{barticle}
\endOrigBibText
\bptok{structpyb}%
\endbibitem

%b11 ###
\bibitem{CoMaMa2010}
\begin{barticle}
\bauthor{\bsnm{Cossette}, \binits{H.}},
\bauthor{\bsnm{Marceau}, \binits{E.}},
\bauthor{\bsnm{Marri}, \binits{F.}}:
\batitle{Analysis of ruin measures for the classical compound {P}oisson risk
 model with dependence}.
\bjtitle{Scand. Actuar. J.}
\bvolume{2010},
\bfpage{221}--\blpage{245}
(\byear{2010}).
\bid{doi={10.1080/03461230903211992}, mr={2732285}}
\end{barticle}
%
\OrigBibText
\begin{barticle}
\bauthor{\bsnm{Cossette}, \binits{H.}},
\bauthor{\bsnm{Marceau}, \binits{E.}},
\bauthor{\bsnm{Marri}, \binits{F.}}:
\batitle{Analysis of ruin measures for the classical compound {P}oisson risk
 model with dependence}.
\bjtitle{Scandinavian Actuarial Journal}
\bvolume{2010},
\bfpage{221}--\blpage{245}
(\byear{2010})
\end{barticle}
\endOrigBibText
\bptok{structpyb}%
\endbibitem

%b12 ###
\bibitem{CoMaMa2011}
\begin{barticle}
\bauthor{\bsnm{Cossette}, \binits{H.}},
\bauthor{\bsnm{Marceau}, \binits{E.}},
\bauthor{\bsnm{Marri}, \binits{F.}}:
\batitle{Constant dividend barrier in a risk model with \xch{a generalized}{ageneralized}
 {F}arlie--{G}umbel--{M}orgenstern copula}.
\bjtitle{Methodol. Comput. Appl. Probab.}
\bvolume{13},
\bfpage{487}--\blpage{510}
(\byear{2011}).
\bid{doi={10.1007/s11009-010-9168-9}, mr={2822392}}
\end{barticle}
%
\OrigBibText
\begin{barticle}
\bauthor{\bsnm{Cossette}, \binits{H.}},
\bauthor{\bsnm{Marceau}, \binits{E.}},
\bauthor{\bsnm{Marri}, \binits{F.}}:
\batitle{Constant dividend barrier in a risk model with ageneralized
 {F}arlie--{G}umbel--{M}orgenstern copula}.
\bjtitle{Methodology and Computing in Applied Probability}
\bvolume{13},
\bfpage{487}--\blpage{510}
(\byear{2011})
\end{barticle}
\endOrigBibText
\bptok{structpyb}%
\endbibitem

%b13 ###
\bibitem{CoMaMa2014}
\begin{barticle}
\bauthor{\bsnm{Cossette}, \binits{H.}},
\bauthor{\bsnm{Marceau}, \binits{E.}},
\bauthor{\bsnm{Marri}, \binits{F.}}:
\batitle{On a compound {P}oisson risk model with dependence and in the presence
 of a constant dividend barrier}.
\bjtitle{Appl. Stoch. Models Bus. Ind.}
\bvolume{30},
\bfpage{82}--\blpage{98}
(\byear{2014}).
\bid{doi={10.1002/asmb.1928}, mr={3191344}}
\end{barticle}
%
\OrigBibText
\begin{barticle}
\bauthor{\bsnm{Cossette}, \binits{H.}},
\bauthor{\bsnm{Marceau}, \binits{E.}},
\bauthor{\bsnm{Marri}, \binits{F.}}:
\batitle{On a compound {P}oisson risk model with dependence and in the presence
 of a constant dividend barrier}.
\bjtitle{Applied Stochastic Models in Business and Industry}
\bvolume{30},
\bfpage{82}--\blpage{98}
(\byear{2014})
\end{barticle}
\endOrigBibText
\bptok{structpyb}%
\endbibitem

%b14 ###
\bibitem{De1957}
\begin{barticle}
\bauthor{\bsnm{De~Finetti}, \binits{B.}}:
\batitle{Su un'impostazione alternativa dell teoria colletiva del rischio}.
\bjtitle{Trans. XV Int. Congr. Actuar.}
\bvolume{2},
\bfpage{433}--\blpage{443}
(\byear{1957})
\end{barticle}
%
\OrigBibText
\begin{barticle}
\bauthor{\bsnm{De~Finetti}, \binits{B.}}:
\batitle{Su un'impostazione alternativa dell teoria colletiva del rischio}.
\bjtitle{Transactions of the XV International Congress of Actuaries}
\bvolume{2},
\bfpage{433}--\blpage{443}
(\byear{1957})
\end{barticle}
\endOrigBibText
\bptok{structpyb}%
\endbibitem

%b15 ###
\bibitem{DeZhDe2012}
\begin{barticle}
\bauthor{\bsnm{Deng}, \binits{C.}},
\bauthor{\bsnm{Zhou}, \binits{J.}},
\bauthor{\bsnm{Deng}, \binits{Y.}}:
\batitle{The {G}erber--{S}hiu discounted penalty function in a delayed renewal
 risk model with multi-layer dividend strategy}.
\bjtitle{Stat. Probab. Lett.}
\bvolume{82},
\bfpage{1648}--\blpage{1656}
(\byear{2012}).
\bid{doi={10.1016/j.spl.2012.05.002}, mr={2951000}}
\end{barticle}
%
\OrigBibText
\begin{barticle}
\bauthor{\bsnm{Deng}, \binits{C.}},
\bauthor{\bsnm{Zhou}, \binits{J.}},
\bauthor{\bsnm{Deng}, \binits{Y.}}:
\batitle{The {G}erber--{S}hiu discounted penalty function in a delayed renewal
 risk model with multi-layer dividend strategy}.
\bjtitle{Statistics and Probability Letters}
\bvolume{82},
\bfpage{1648}--\blpage{1656}
(\byear{2012})
\end{barticle}
\endOrigBibText
\bptok{structpyb}%
\endbibitem

%b16 ###
\bibitem{GeSh1998}
\begin{barticle}
\bauthor{\bsnm{Gerber}, \binits{H.U.}},
\bauthor{\bsnm{Shiu}, \binits{E.S.W.}}:
\batitle{On the time value of ruin}.
\bjtitle{N. Am. Actuar. J.}
\bvolume{2},
\bfpage{48}--\blpage{72}
(\byear{1998}).
\bid{doi={10.1080/10920277.1998.10595671}, mr={1988433}}
\end{barticle}
%
\OrigBibText
\begin{barticle}
\bauthor{\bsnm{Gerber}, \binits{H.U.}},
\bauthor{\bsnm{Shiu}, \binits{E.S.W.}}:
\batitle{On the time value of ruin}.
\bjtitle{North American Actuarial Journal}
\bvolume{2},
\bfpage{48}--\blpage{72}
(\byear{1998})
\end{barticle}
\endOrigBibText
\bptok{structpyb}%
\endbibitem

%b17 ###
\bibitem{GeSh2005}
\begin{barticle}
\bauthor{\bsnm{Gerber}, \binits{H.U.}},
\bauthor{\bsnm{Shiu}, \binits{E.S.W.}}:
\batitle{The time value of ruin in a {S}parre {A}ndersen model}.
\bjtitle{N. Am. Actuar. J.}
\bvolume{9},
\bfpage{49}--\blpage{69}
(\byear{2005}).
\bid{doi={10.1080/10920277.2005.\\10596197}, mr={2160108}}
\end{barticle}
%
\OrigBibText
\begin{barticle}
\bauthor{\bsnm{Gerber}, \binits{H.U.}},
\bauthor{\bsnm{Shiu}, \binits{E.S.W.}}:
\batitle{The time value of ruin in a {S}parre {A}ndersen model}.
\bjtitle{North American Actuarial Journal}
\bvolume{9},
\bfpage{49}--\blpage{69}
(\byear{2005})
\end{barticle}
\endOrigBibText
\bptok{structpyb}%
\endbibitem

%b18 ###
\bibitem{He2014}
\begin{barticle}
\bauthor{\bsnm{Heilpern}, \binits{S.}}:
\batitle{Ruin measures for a compound {P}oisson risk model with dependence
 based on the {S}pearman copula and the exponential claim sizes}.
\bjtitle{Insur. Math. Econ.}
\bvolume{59},
\bfpage{251}--\blpage{257}
(\byear{2014}).
\bid{doi={10.1016/j.insmatheco.2014.10.006}, mr={3283226}}
\end{barticle}
%
\OrigBibText
\begin{barticle}
\bauthor{\bsnm{Heilpern}, \binits{S.}}:
\batitle{Ruin measures for a compound {P}oisson risk model with dependence
 based on the {S}pearman copula and the exponential claim sizes}.
\bjtitle{Insurance: Mathematics and Economics}
\bvolume{59},
\bfpage{251}--\blpage{257}
(\byear{2014})
\end{barticle}
\endOrigBibText
\bptok{structpyb}%
\endbibitem

%b19 ###
\bibitem{JiYaLi2012}
\begin{barticle}
\bauthor{\bsnm{Jiang}, \binits{W.}},
\bauthor{\bsnm{Yang}, \binits{Z.}},
\bauthor{\bsnm{Li}, \binits{X.}}:
\batitle{The discounted penalty function with multi-layer dividend strategy in
 the phase-type risk model}.
\bjtitle{Stat. Probab. Lett.}
\bvolume{82},
\bfpage{1358}--\blpage{1366}
(\byear{2012}).
\bid{doi={10.1016/j.spl.2012.03.012}, mr={2929787}}
\end{barticle}
%
\OrigBibText
\begin{barticle}
\bauthor{\bsnm{Jiang}, \binits{W.}},
\bauthor{\bsnm{Yang}, \binits{Z.}},
\bauthor{\bsnm{Li}, \binits{X.}}:
\batitle{The discounted penalty function with multi-layer dividend strategy in
 the phase-type risk model}.
\bjtitle{Statistics and Probability Letters}
\bvolume{82},
\bfpage{1358}--\blpage{1366}
(\byear{2012})
\end{barticle}
\endOrigBibText
\bptok{structpyb}%
\endbibitem

%b20 ###
\bibitem{La2008}
\begin{barticle}
\bauthor{\bsnm{Landriault}, \binits{D.}}:
\batitle{Constant dividend barrier in a risk model with interclaim-dependent
 claim sizes}.
\bjtitle{Insur. Math. Econ.}
\bvolume{42},
\bfpage{31}--\blpage{38}
(\byear{2008}).
\bid{doi={10.1016/\\j.insmatheco.2006.12.002}, mr={2392066}}
\end{barticle}
%
\OrigBibText
\begin{barticle}
\bauthor{\bsnm{Landriault}, \binits{D.}}:
\batitle{Constant dividend barrier in a risk model with interclaim-dependent
 claim sizes}.
\bjtitle{Insurance: Mathematics and Economics}
\bvolume{42},
\bfpage{31}--\blpage{38}
(\byear{2008})
\end{barticle}
\endOrigBibText
\bptok{structpyb}%
\endbibitem

%b21 ###
\bibitem{LiMa2016}
\begin{barticle}
\bauthor{\bsnm{Li}, \binits{B.}},
\bauthor{\bsnm{Ma}, \binits{S.}}:
\batitle{Markov-dependent risk model with multi-layer dividend strategy and
 investment interest under absolute ruin}.
\bjtitle{Math. Finance}
\bvolume{6},
\bfpage{260}--\blpage{268}
(\byear{2016})
\end{barticle}
%
\OrigBibText
\begin{barticle}
\bauthor{\bsnm{Li}, \binits{B.}},
\bauthor{\bsnm{Ma}, \binits{S.}}:
\batitle{Markov-dependent risk model with multi-layer dividend strategy and
 investment interest under absolute ruin}.
\bjtitle{Journal of Mathematical Finance}
\bvolume{6},
\bfpage{260}--\blpage{268}
(\byear{2016})
\end{barticle}
\endOrigBibText
\bptok{structpyb}%
\endbibitem

%b22 ###
\bibitem{LiWuSo2009}
\begin{barticle}
\bauthor{\bsnm{Li}, \binits{B.}},
\bauthor{\bsnm{Wu}, \binits{R.}},
\bauthor{\bsnm{Song}, \binits{M.}}:
\batitle{A renewal jump-diffusion process with threshold dividend strategy}.
\bjtitle{J. Comput. Appl. Math.}
\bvolume{228},
\bfpage{41}--\blpage{55}
(\byear{2009}).
\bid{doi={10.1016/\\j.cam.2008.08.046}, mr={2517685}}
\end{barticle}
%
\OrigBibText
\begin{barticle}
\bauthor{\bsnm{Li}, \binits{B.}},
\bauthor{\bsnm{Wu}, \binits{R.}},
\bauthor{\bsnm{Song}, \binits{M.}}:
\batitle{A renewal jump-diffusion process with threshold dividend strategy}.
\bjtitle{Journal of Computational and Applied Mathematics}
\bvolume{228},
\bfpage{41}--\blpage{55}
(\byear{2009})
\end{barticle}
\endOrigBibText
\bptok{structpyb}%
\endbibitem

%b23 ###
\bibitem{LiGa2004}
\begin{barticle}
\bauthor{\bsnm{Li}, \binits{S.}},
\bauthor{\bsnm{Garrido}, \binits{J.}}:
\batitle{On a class of renewal risk models with a constant dividend barrier}.
\bjtitle{Insur. Math. Econ.}
\bvolume{35},
\bfpage{691}--\blpage{701}
(\byear{2004}).
\bid{doi={10.1016/j.insmatheco.\\2004.08.004}, mr={2106143}}
\end{barticle}
%
\OrigBibText
\begin{barticle}
\bauthor{\bsnm{Li}, \binits{S.}},
\bauthor{\bsnm{Garrido}, \binits{J.}}:
\batitle{On a class of renewal risk models with a constant dividend barrier}.
\bjtitle{Insurance: Mathematics and Economics}
\bvolume{35},
\bfpage{691}--\blpage{701}
(\byear{2004})
\end{barticle}
\endOrigBibText
\bptok{structpyb}%
\endbibitem

%b24 ###
\bibitem{LiPa2006}
\begin{barticle}
\bauthor{\bsnm{Lin}, \binits{X.S.}},
\bauthor{\bsnm{Pavlova}, \binits{K.P.}}:
\batitle{The compound {P}oisson risk model with a threshold dividend strategy}.
\bjtitle{Insur. Math. Econ.}
\bvolume{38},
\bfpage{57}--\blpage{80}
(\byear{2006}).
\bid{doi={10.1016/\\j.insmatheco.2005.08.001}, mr={2197303}}
\end{barticle}
%
\OrigBibText
\begin{barticle}
\bauthor{\bsnm{Lin}, \binits{X.S.}},
\bauthor{\bsnm{Pavlova}, \binits{K.P.}}:
\batitle{The compound {P}oisson risk model with a threshold dividend strategy}.
\bjtitle{Insurance: Mathematics and Economics}
\bvolume{38},
\bfpage{57}--\blpage{80}
(\byear{2006})
\end{barticle}
\endOrigBibText
\bptok{structpyb}%
\endbibitem

%b25 ###
\bibitem{LiSe2008}
\begin{barticle}
\bauthor{\bsnm{Lin}, \binits{X.S.}},
\bauthor{\bsnm{Sendova}, \binits{K.P.}}:
\batitle{The compound {P}oisson risk model with multiple thresholds}.
\bjtitle{Insur. Math. Econ.}
\bvolume{42},
\bfpage{617}--\blpage{627}
(\byear{2008}).
\bid{doi={10.1016/j.insmatheco.\\2007.06.008}, mr={2404318}}
\end{barticle}
%
\OrigBibText
\begin{barticle}
\bauthor{\bsnm{Lin}, \binits{X.S.}},
\bauthor{\bsnm{Sendova}, \binits{K.P.}}:
\batitle{The compound {P}oisson risk model with multiple thresholds}.
\bjtitle{Insurance: Mathematics and Economics}
\bvolume{42},
\bfpage{617}--\blpage{627}
(\byear{2008})
\end{barticle}
\endOrigBibText
\bptok{structpyb}%
\endbibitem

%b26 ###
\bibitem{LiWiDr2003}
\begin{barticle}
\bauthor{\bsnm{Lin}, \binits{X.S.}},
\bauthor{\bsnm{Willmot}, \binits{G.E.}},
\bauthor{\bsnm{Drekic}, \binits{S.}}:
\batitle{The classical risk model with a constant dividend barrier: analysis of
 the {G}erber-{S}hiu discounted penalty function}.
\bjtitle{Insur. Math. Econ.}
\bvolume{33},
\bfpage{551}--\blpage{566}
(\byear{2003}).
\bid{doi={10.1016/j.insmatheco.2003.08.004}, mr={2021233}}
\end{barticle}
%
\OrigBibText
\begin{barticle}
\bauthor{\bsnm{Lin}, \binits{X.S.}},
\bauthor{\bsnm{Willmot}, \binits{G.E.}},
\bauthor{\bsnm{Drekic}, \binits{S.}}:
\batitle{The classical risk model with a constant dividend barrier: analysis of
 the {G}erber-{S}hiu discounted penalty function}.
\bjtitle{Insurance: Mathematics and Economics}
\bvolume{33},
\bfpage{551}--\blpage{566}
(\byear{2003})
\end{barticle}
\endOrigBibText
\bptok{structpyb}%
\endbibitem

%b27 ###
\bibitem{LiLiPe2014}
\begin{barticle}
\bauthor{\bsnm{Liu}, \binits{D.}},
\bauthor{\bsnm{Liu}, \binits{Z.}},
\bauthor{\bsnm{Peng}, \binits{D.}}:
\batitle{The {G}erber--{S}hiu expected penalty function for the risk model with
 dependence and a \xch{constant}{constand} dividend barrier}.
\bjtitle{Abstr. Appl. Anal.}
\bvolume{2014},
\bnumber{730174},
\bfpage{1}--\blpage{7}
(\byear{2014}).
\bid{doi={10.1155/2014/730174}, mr={3246356}}
\end{barticle}
%
\OrigBibText
\begin{barticle}
\bauthor{\bsnm{Liu}, \binits{D.}},
\bauthor{\bsnm{Liu}, \binits{Z.}},
\bauthor{\bsnm{Peng}, \binits{D.}}:
\batitle{The {G}erber--{S}hiu expected penalty function for the risk model with
 dependence and a constand dividend barrier}.
\bjtitle{Abstract and Applied Analysis}
\bvolume{2014},
\bfpage{730174}--\blpage{7}
(\byear{2014})
\end{barticle}
\endOrigBibText
\bptok{structpyb}%
\endbibitem

%b28 ###
\bibitem{MiRa2016}
\begin{bbook}
\bauthor{\bsnm{Mishura}, \binits{Y.}},
\bauthor{\bsnm{Ragulina}, \binits{O.}}:
\bbtitle{Ruin Probabilities: Smoothness, Bounds, Supermartingale Approach}.
\bpublisher{ISTE Press -- Elsevier},
\blocation{London}
(\byear{2016}).
\bid{mr={3643478}}
\end{bbook}
%
\OrigBibText
\begin{bbook}
\bauthor{\bsnm{Mishura}, \binits{Y.}},
\bauthor{\bsnm{Ragulina}, \binits{O.}}:
\bbtitle{Ruin Probabilities: Smoothness, Bounds, Supermartingale Approach}.
\bpublisher{ISTE Press -- Elsevier},
\blocation{London}
(\byear{2016})
\end{bbook}
\endOrigBibText
\bptok{structpyb}%
\endbibitem

%b29 ###
\bibitem{MiRaSt2014}
\begin{barticle}
\bauthor{\bsnm{Mishura}, \binits{Y.}},
\bauthor{\bsnm{Ragulina}, \binits{O.}},
\bauthor{\bsnm{Stroev}, \binits{O.}}:
\batitle{Practical approaches to the estimation of the ruin probability in a
 risk model with additional funds}.
\bjtitle{Mod. Stoch. Theory Appl.}
\bvolume{1},
\bfpage{167}--\blpage{180}
(\byear{2014}).
\bid{doi={10.15559/15-VMSTA18}, mr={3316485}}
\end{barticle}
%
\OrigBibText
\begin{barticle}
\bauthor{\bsnm{Mishura}, \binits{Y.}},
\bauthor{\bsnm{Ragulina}, \binits{O.}},
\bauthor{\bsnm{Stroev}, \binits{O.}}:
\batitle{Practical approaches to the estimation of the ruin probability in a
 risk model with additional funds}.
\bjtitle{Modern Stochastics: Theory and Applications}
\bvolume{1},
\bfpage{167}--\blpage{180}
(\byear{2014})
\end{barticle}
\endOrigBibText
\bptok{structpyb}%
\endbibitem

%b30 ###
\bibitem{MiRaSt2015}
\begin{barticle}
\bauthor{\bsnm{Mishura}, \binits{Y.S.}},
\bauthor{\bsnm{Ragulina}, \binits{O.Y.}},
\bauthor{\bsnm{Stroev}, \binits{O.M.}}:
\batitle{Analytic property of infinite-horizon survival probability in a risk
 model with additional funds}.
\bjtitle{Theory Probab. Math. Stat.}
\bvolume{91},
\bfpage{131}--\blpage{143}
(\byear{2015})
\end{barticle}
%
\OrigBibText
\begin{barticle}
\bauthor{\bsnm{Mishura}, \binits{Y.S.}},
\bauthor{\bsnm{Ragulina}, \binits{O.Y.}},
\bauthor{\bsnm{Stroev}, \binits{O.M.}}:
\batitle{Analytic property of infinite-horizon survival probability in a risk
 model with additional funds}.
\bjtitle{Theory of Probability and Mathematical Statistics}
\bvolume{91},
\bfpage{131}--\blpage{143}
(\byear{2015})
\end{barticle}
\endOrigBibText
\bptok{structpyb}%
\endbibitem

%b31 ###
\bibitem{MiSeTs2010}
\begin{barticle}
\bauthor{\bsnm{Mitric}, \binits{I.-R.}},
\bauthor{\bsnm{Sendova}, \binits{K.P.}},
\bauthor{\bsnm{Tsai}, \binits{C.C.-L.}}:
\batitle{On a multi-threshold compound \xch{Poisson}{poisson} process perturbed by diffusion}.
\bjtitle{Stat. Probab. Lett.}
\bvolume{80},
\bfpage{366}--\blpage{375}
(\byear{2010}).
\bid{doi={10.1016/j.spl.2009.11.012}, mr={2593575}}
\end{barticle}
%
\OrigBibText
\begin{barticle}
\bauthor{\bsnm{Mitric}, \binits{I.-R.}},
\bauthor{\bsnm{Sendova}, \binits{K.P.}},
\bauthor{\bsnm{Tsai}, \binits{C.C.-L.}}:
\batitle{On a multi-threshold compound poisson process perturbed by diffusion}.
\bjtitle{Statistics and Probability Letters}
\bvolume{80},
\bfpage{366}--\blpage{375}
(\byear{2010})
\end{barticle}
\endOrigBibText
\bptok{structpyb}%
\endbibitem

%b32 ###
\bibitem{Ra2017}
\begin{barticle}
\bauthor{\bsnm{Ragulina}, \binits{O.}}:
\batitle{The risk model with stochastic premiums, dependence and a threshold
 dividend strategy}.
\bjtitle{Mod. Stoch. Theory Appl.}
\bvolume{4},
\bfpage{315}--\blpage{351}
(\byear{2017}).
\bid{doi={\\10.15559/17-vmsta89}, mr={3739013}}
\end{barticle}
%
\OrigBibText
\begin{barticle}
\bauthor{\bsnm{Ragulina}, \binits{O.}}:
\batitle{The risk model with stochastic premiums, dependence and a threshold
 dividend strategy}.
\bjtitle{Modern Stochastics: Theory and Applications}
\bvolume{4},
\bfpage{315}--\blpage{351}
(\byear{2017})
\end{barticle}
\endOrigBibText
\bptok{structpyb}%
\endbibitem

%b33 ###
\bibitem{RoScScTe1999}
\begin{bbook}
\bauthor{\bsnm{Rolski}, \binits{T.}},
\bauthor{\bsnm{Schmidli}, \binits{H.}},
\bauthor{\bsnm{Schmidt}, \binits{V.}},
\bauthor{\bsnm{Teugels}, \binits{J.}}:
\bbtitle{Stochastic Processes for Insurance and Finance}.
\bpublisher{John Wiley \& Sons},
\blocation{Chichester}
(\byear{1999}).
\bid{doi={\\10.1002/9780470317044}, mr={1680267}}
\end{bbook}
%
\OrigBibText
\begin{bbook}
\bauthor{\bsnm{Rolski}, \binits{T.}},
\bauthor{\bsnm{Schmidli}, \binits{H.}},
\bauthor{\bsnm{Schmidt}, \binits{V.}},
\bauthor{\bsnm{Teugels}, \binits{J.}}:
\bbtitle{Stochastic Processes for Insurance and Finance}.
\bpublisher{John Wiley \& Sons},
\blocation{Chichester}
(\byear{1999})
\end{bbook}
\endOrigBibText
\bptok{structpyb}%
\endbibitem

%b34 ###
\bibitem{Sc2008}
\begin{bbook}
\bauthor{\bsnm{Schmidli}, \binits{H.}}:
\bbtitle{Stochastic Control in Insurance}.
\bpublisher{Springer},
\blocation{London}
(\byear{2008}).
\bid{mr={2371646}}
\end{bbook}
%
\OrigBibText
\begin{bbook}
\bauthor{\bsnm{Schmidli}, \binits{H.}}:
\bbtitle{Stochastic Control in Insurance}.
\bpublisher{Springer},
\blocation{London}
(\byear{2008})
\end{bbook}
\endOrigBibText
\bptok{structpyb}%
\endbibitem

%b35 ###
\bibitem{ShLiZh2013}
\begin{barticle}
\bauthor{\bsnm{Shi}, \binits{Y.}},
\bauthor{\bsnm{Liu}, \binits{P.}},
\bauthor{\bsnm{Zhang}, \binits{C.}}:
\batitle{On the compound \xch{Poisson}{poisson} risk model with dependence and a threshold
 dividend strategy}.
\bjtitle{Stat. Probab. Lett.}
\bvolume{83},
\bfpage{1998}--\blpage{2006}
(\byear{2013}).
\bid{doi={10.1016/j.spl.2013.05.008}, mr={3079035}}
\end{barticle}
%
\OrigBibText
\begin{barticle}
\bauthor{\bsnm{Shi}, \binits{Y.}},
\bauthor{\bsnm{Liu}, \binits{P.}},
\bauthor{\bsnm{Zhang}, \binits{C.}}:
\batitle{On the compound poisson risk model with dependence and a threshold
 dividend strategy}.
\bjtitle{Statistics and Probability Letters}
\bvolume{83},
\bfpage{1998}--\blpage{2006}
(\byear{2013})
\end{barticle}
\endOrigBibText
\bptok{structpyb}%
\endbibitem

%b36 ###
\bibitem{Su2005}
\begin{barticle}
\bauthor{\bsnm{Sun}, \binits{L.-J.}}:
\batitle{The expected discounted penalty at ruin in the {E}rlang(2) risk
 process}.
\bjtitle{Stat. Probab. Lett.}
\bvolume{72},
\bfpage{205}--\blpage{217}
(\byear{2005}).
\bid{doi={10.1016/j.spl.2004.12.015}, mr={2137163}}
\end{barticle}
%
\OrigBibText
\begin{barticle}
\bauthor{\bsnm{Sun}, \binits{L.-J.}}:
\batitle{The expected discounted penalty at ruin in the {E}rlang(2) risk
 process}.
\bjtitle{Statistics and Probability Letters}
\bvolume{72},
\bfpage{205}--\blpage{217}
(\byear{2005})
\end{barticle}
\endOrigBibText
\bptok{structpyb}%
\endbibitem

%b37 ###
\bibitem{Wa2015}
\begin{barticle}
\bauthor{\bsnm{Wang}, \binits{W.}}:
\batitle{The perturbed {S}parre {A}ndersen model with interest and a threshold
 dividend strategy}.
\bjtitle{Methodol. Comput. Appl. Probab.}
\bvolume{17},
\bfpage{251}--\blpage{283}
(\byear{2015}).
\bid{doi={10.1007/s11009-013-9332-0}, mr={3343407}}
\end{barticle}
%
\OrigBibText
\begin{barticle}
\bauthor{\bsnm{Wang}, \binits{W.}}:
\batitle{The perturbed {S}parre {A}ndersen model with interest and a threshold
 dividend strategy}.
\bjtitle{Methodology and Computing in Applied Probability}
\bvolume{17},
\bfpage{251}--\blpage{283}
(\byear{2015})
\end{barticle}
\endOrigBibText
\bptok{structpyb}%
\endbibitem

%b38 ###
\bibitem{XiZo2017}
\begin{barticle}
\bauthor{\bsnm{Xie}, \binits{J.-H.}},
\bauthor{\bsnm{Zou}, \binits{W.}}:
\batitle{On the expected discounted penalty function for a risk model with
 dependence under a multi-layer dividend strategy}.
\bjtitle{Commun. Statis.-Theor. Methods}
\bvolume{46},
\bfpage{1898}--\blpage{1915}
(\byear{2017}).
\bid{doi={10.1080/03610926.2015.1030424}, mr={3574735}}
\end{barticle}
%
\OrigBibText
\begin{barticle}
\bauthor{\bsnm{Xie}, \binits{J.-H.}},
\bauthor{\bsnm{Zou}, \binits{W.}}:
\batitle{On the expected discounted penalty function for a risk model with
 dependence under a multi-layer dividend strategy}.
\bjtitle{Cmmunications in Statistics -- Theory and Methods}
\bvolume{46},
\bfpage{1898}--\blpage{1915}
(\byear{2017})
\end{barticle}
\endOrigBibText
\bptok{structpyb}%
\endbibitem

%b39 ###
\bibitem{XiWu2006}
\begin{barticle}
\bauthor{\bsnm{Xing}, \binits{Y.}},
\bauthor{\bsnm{Wu}, \binits{R.}}:
\batitle{Moments of the time of ruin, surplus before ruin and the deficit at
 ruin in the {E}rlang({N}) risk process}.
\bjtitle{Acta Math. Appl. Sin. Engl. Ser.}
\bvolume{22},
\bfpage{599}--\blpage{606}
(\byear{2006}).
\bid{doi={10.1007/s10255-006-0333-4}, mr={2248523}}
\end{barticle}
%
\OrigBibText
\begin{barticle}
\bauthor{\bsnm{Xing}, \binits{Y.}},
\bauthor{\bsnm{Wu}, \binits{R.}}:
\batitle{Moments of the time of ruin, surplus before ruin and the deficit at
 ruin in the {E}rlang({N}) risk process}.
\bjtitle{Acta Mathematicae Applicatae Sinica, English Series}
\bvolume{22},
\bfpage{599}--\blpage{606}
(\byear{2006})
\end{barticle}
\endOrigBibText
\bptok{structpyb}%
\endbibitem

%b40 ###
\bibitem{YaZh2008}
\begin{barticle}
\bauthor{\bsnm{Yang}, \binits{H.}},
\bauthor{\bsnm{Zhang}, \binits{Z.}}:
\batitle{{G}erber--{S}hiu discounted penalty function in a {S}parre {A}ndersen
 model with multi-layer dividend strategy}.
\bjtitle{Insur. Math. Econ.}
\bvolume{42},
\bfpage{984}--\blpage{991}
(\byear{2008}).
\bid{doi={10.1016/j.insmatheco.2007.11.004}, mr={2435369}}
\end{barticle}
%
\OrigBibText
\begin{barticle}
\bauthor{\bsnm{Yang}, \binits{H.}},
\bauthor{\bsnm{Zhang}, \binits{Z.}}:
\batitle{{G}erber--{S}hiu discounted penalty function in a {S}parre {A}ndersen
 model with multi-layer dividend strategy}.
\bjtitle{Insurance: Mathematics and Economics}
\bvolume{42},
\bfpage{984}--\blpage{991}
(\byear{2008})
\end{barticle}
\endOrigBibText
\bptok{structpyb}%
\endbibitem

%b41 ###
\bibitem{YaZh2009_1}
\begin{barticle}
\bauthor{\bsnm{Yang}, \binits{H.}},
\bauthor{\bsnm{Zhang}, \binits{Z.}}:
\batitle{On a perturbed {S}parre {A}ndersen risk model with multi-layer
 dividend strategy}.
\bjtitle{J. Comput. Appl. Math.}
\bvolume{232},
\bfpage{612}--\blpage{624}
(\byear{2009}).
\bid{doi={\\10.1016/j.cam.2009.06.032}, mr={2555426}}
\end{barticle}
%
\OrigBibText
\begin{barticle}
\bauthor{\bsnm{Yang}, \binits{H.}},
\bauthor{\bsnm{Zhang}, \binits{Z.}}:
\batitle{On a perturbed {S}parre {A}ndersen risk model with multi-layer
 dividend strategy}.
\bjtitle{Journal of Computtional and Applied Mathematics}
\bvolume{232},
\bfpage{612}--\blpage{624}
(\byear{2009})
\end{barticle}
\endOrigBibText
\bptok{structpyb}%
\endbibitem

%b42 ###
\bibitem{YaZh2009_2}
\begin{barticle}
\bauthor{\bsnm{Yang}, \binits{H.}},
\bauthor{\bsnm{Zhang}, \binits{Z.}}:
\batitle{The perturbed compound {P}oisson risk model with multi-layer dividend
 strategy}.
\bjtitle{Stat. Probab. Lett.}
\bvolume{79},
\bfpage{70}--\blpage{78}
(\byear{2009}).
\bid{doi={10.1016/\\j.spl.2008.07.017}, mr={2655109}}
\end{barticle}
%
\OrigBibText
\begin{barticle}
\bauthor{\bsnm{Yang}, \binits{H.}},
\bauthor{\bsnm{Zhang}, \binits{Z.}}:
\batitle{The perturbed compound {P}oisson risk model with multi-layer dividend
 strategy}.
\bjtitle{Statistics and Probability Letters}
\bvolume{79},
\bfpage{70}--\blpage{78}
(\byear{2009})
\end{barticle}
\endOrigBibText
\bptok{structpyb}%
\endbibitem

%b43 ###
\bibitem{YaZhLa2008}
\begin{barticle}
\bauthor{\bsnm{Yang}, \binits{H.}},
\bauthor{\bsnm{Zhang}, \binits{Z.}},
\bauthor{\bsnm{Lan}, \binits{C.}}:
\batitle{On the time value of absolute ruin for a \xch{multi-layer}{mulyi-layer} compound
 {P}oisson model under interest force}.
\bjtitle{Stat. Probab. Lett.}
\bvolume{78},
\bfpage{1835}--\blpage{1845}
(\byear{2008}).
\bid{doi={10.1016/j.spl.2008.01.038}, mr={2453921}}
\end{barticle}
%
\OrigBibText
\begin{barticle}
\bauthor{\bsnm{Yang}, \binits{H.}},
\bauthor{\bsnm{Zhang}, \binits{Z.}},
\bauthor{\bsnm{Lan}, \binits{C.}}:
\batitle{On the time value of absolute ruin for a mulyi-layer compound
 {P}oisson model under interest force}.
\bjtitle{Statistics and Probability Letters}
\bvolume{78},
\bfpage{1835}--\blpage{1845}
(\byear{2008})
\end{barticle}
\endOrigBibText
\bptok{structpyb}%
\endbibitem

%b44 ###
\bibitem{Yi2012}
\begin{barticle}
\bauthor{\bsnm{Yin}, \binits{J.}}:
\batitle{On a dual model with multi-layer dividend strategy under stochastic
 interest}.
\bjtitle{Adv. Mater. Res.}
\bvolume{422},
\bfpage{775}--\blpage{778}
(\byear{2012})
\end{barticle}
%
\OrigBibText
\begin{barticle}
\bauthor{\bsnm{Yin}, \binits{J.}}:
\batitle{On a dual model with multi-layer dividend strategy under stochastic
 interest}.
\bjtitle{Advanced Materials Research}
\bvolume{422},
\bfpage{775}--\blpage{778}
(\byear{2012})
\end{barticle}
\endOrigBibText
\bptok{structpyb}%
\endbibitem

%b45 ###
\bibitem{YoXi2012}
\begin{barticle}
\bauthor{\bsnm{Yong}, \binits{W.}},
\bauthor{\bsnm{Xiang}, \binits{H.}}:
\batitle{Differential equations for ruin probability in a special risk model
 with {FGM} copula for the claim size and the inter-claim time}.
\bjtitle{J. Inequal. Appl.}
\bvolume{2012},
\bnumber{156},
\bfpage{1}--\blpage{13}
(\byear{2012}).
\bid{doi={10.1186/1029-242X-2012-156}, mr={3085740}}
\end{barticle}
%
\OrigBibText
\begin{barticle}
\bauthor{\bsnm{Yong}, \binits{W.}},
\bauthor{\bsnm{Xiang}, \binits{H.}}:
\batitle{Differential equations for ruin probability in a special risk model
 with {FGM} copula for the claim size and the inter-claim time}.
\bjtitle{Journal of Inequalities and Applications}
\bvolume{2012},
\bfpage{156}--\blpage{13}
(\byear{2012})
\end{barticle}
\endOrigBibText
\bptok{structpyb}%
\endbibitem

%b46 ###
\bibitem{ZhYa2011_2}
\begin{barticle}
\bauthor{\bsnm{Zhang}, \binits{Z.}},
\bauthor{\bsnm{Yang}, \binits{H.}}:
\batitle{The compound {P}oisson risk model with dependence under a multi-layer
 dividend strategy}.
\bjtitle{Appl. Math. J. Chin. Univ.}
\bvolume{26},
\bfpage{1}--\blpage{13}
(\byear{2011}).
\bid{doi={10.1007/s11766-011-2279-4}, mr={2776860}}
\end{barticle}
%
\OrigBibText
\begin{barticle}
\bauthor{\bsnm{Zhang}, \binits{Z.}},
\bauthor{\bsnm{Yang}, \binits{H.}}:
\batitle{The compound {P}oisson risk model with dependence under a multi-layer
 dividend strategy}.
\bjtitle{Applied Mathematics -- A Journal of Chinese Universities}
\bvolume{26},
\bfpage{1}--\blpage{13}
(\byear{2011})
\end{barticle}
\endOrigBibText
\bptok{structpyb}%
\endbibitem

%b47 ###
\bibitem{ZhYa2011_1}
\begin{barticle}
\bauthor{\bsnm{Zhang}, \binits{Z.}},
\bauthor{\bsnm{Yang}, \binits{H.}}:
\batitle{{G}erber--{S}hiu analysis in a perturbed risk model with dependence
 between claim sizes and interclaim times}.
\bjtitle{J. Comput. Appl. Math.}
\bvolume{235},
\bfpage{1189}--\blpage{1204}
(\byear{2011}).
\bid{doi={10.1016/j.cam.2010.08.003}, mr={2728058}}
\end{barticle}
%
\OrigBibText
\begin{barticle}
\bauthor{\bsnm{Zhang}, \binits{Z.}},
\bauthor{\bsnm{Yang}, \binits{H.}}:
\batitle{{G}erber--{S}hiu analysis in a perturbed risk model with dependence
 between claim sizes and interclaim times}.
\bjtitle{Journal of Computational and Applied Mathematics}
\bvolume{235},
\bfpage{1189}--\blpage{1204}
(\byear{2011})
\end{barticle}
\endOrigBibText
\bptok{structpyb}%
\endbibitem

%b48 ###
\bibitem{ZhXiDe2015}
\begin{barticle}
\bauthor{\bsnm{Zhou}, \binits{Z.}},
\bauthor{\bsnm{Xiao}, \binits{H.}},
\bauthor{\bsnm{Deng}, \binits{Y.}}:
\batitle{Markov-dependent risk model with multi-layer dividend strategy}.
\bjtitle{Appl. Math. Comput.}
\bvolume{252},
\bfpage{273}--\blpage{286}
(\byear{2015}).
\bid{doi={\\10.1016/j.amc.2014.12.016}, mr={3305106}}
\end{barticle}
%
\OrigBibText
\begin{barticle}
\bauthor{\bsnm{Zhou}, \binits{Z.}},
\bauthor{\bsnm{Xiao}, \binits{H.}},
\bauthor{\bsnm{Deng}, \binits{Y.}}:
\batitle{Markov-dependent risk model with multi-layer dividend strategy}.
\bjtitle{Applied Mathematics and Computation}
\bvolume{252},
\bfpage{273}--\blpage{286}
(\byear{2015})
\end{barticle}
\endOrigBibText
\bptok{structpyb}%
\endbibitem

\end{thebibliography}
\end{document}